\documentclass{SCAEOL}
\numberwithin{equation}{section}

\usepackage{multirow}
\usepackage{enumerate}
\usepackage{graphicx}
\usepackage{epstopdf}
\usepackage{hyperref}

\DeclareMathOperator{\diag}{diag}
\DeclareMathOperator{\spn}{span}
\newcommand\x{\mathbf{x}}
\newcommand\y{\mathbf{y}}

\begin{document}

\Year{2020} %
\Month{Apr}
\Vol{XX} %
\No{XX} %
\BeginPage{1} %
\EndPage{XX} %
\AuthorMark{Liu W et al.}

\title[Convergence analysis of a spectral-Galerkin-type search extension method]{Convergence analysis of a spectral-Galerkin-type search extension method for finding multiple solutions to semilinear problems\footnotetext{This work was funded by National Natural Science Foundation of China (Grant No. 11771138, 11171104, 11971007, 11601148) and the Opening Fund for Innovation Platform of Education  Department in Hunan Province (Grant No. 18K025).}}{Published in SCIENTIA SINICA Mathematica (in Chinese).\\[1ex]
{\bf Citation:} Liu W, Xie Z Q, Yuan Y J. Convergence analysis of a spectral-Galerkin-type search extension method for finding multiple solutions to semilinear problems (in Chinese). Sci Sin Math, 2021, 51: 1407–1431, doi: \href{http://doi.org/10.1360/SCM-2019-0357}{10.1360/SCM-2019-0357}}


\author{Wei Liu}{}
\author{Ziqing Xie}{}
\author{Yongjun Yuan}{Corresponding author.}

\address{LCSM (MOE) and School of Mathematics and Statistics, Hunan Normal University, Changsha {410081}, P.R. China}
\Emails{~\tt wliu.hunnu@foxmail.com, ziqingxie@hunnu.edu.cn, yuanyongjun0301@163.com}
\maketitle


{\begin{center}
\parbox{14.5cm}{
\begin{abstract}
In this paper, we develop an efficient spectral-Galerkin-type search extension method (SGSEM) for finding multiple solutions to semilinear elliptic boundary value problems. This method constructs effective initial data for multiple solutions based on the linear combinations of some eigenfunctions of the corresponding linear eigenvalue problem, and thus takes full advantage of the traditional search extension method in constructing initials for multiple solutions. Meanwhile, it possesses a low computational cost and high accuracy due to the employment of an interpolated coefficient Legendre-Galerkin spectral discretization. By applying the Schauder's fixed point theorem and other technical strategies, the existence and spectral convergence of the numerical solution corresponding to a specified true solution are rigorously proved. In addition, the uniqueness of the numerical solution in a sufficiently small neighborhood of each specified true solution is strictly verified. Numerical results demonstrate the feasibility and efficiency of our algorithm and present different types of multiple solutions.
\vspace{-3mm}
\end{abstract}}
\end{center}}



\renewcommand{\baselinestretch}{1.2}
\begin{center} \renewcommand{\arraystretch}{1.5}
{\begin{tabular}{lp{0.8\textwidth}} \hline \scriptsize
{\bf ~~~~~~~~~~}\!\!\!\!&\scriptsize
\end{tabular}}\end{center}

\baselineskip 11pt\parindent=10.8pt  \wuhao
\section{Introduction}
\label{sec:1}
Consider finding multiple solutions to the following two-dimensional semilinear elliptic Dirichlet boundary value problem
\begin{equation}\label{HE1}
 \left\{\begin{array}{rll}
  \mathcal{L}(u)=-\Delta u-f(u)=0 &\mbox{in}& \Omega,\\
  u=0 & \mbox{on} & \partial\Omega,
 \end{array}\right.
\end{equation}
where $\Delta$ is the Laplacian, $\Omega$ is a rectangular domain in $\mathbb{R}^2$, and $f:\mathbb{R}\to\mathbb{R}$ is an appropriately smooth nonlinear function. The model problem \eqref{HE1} is widely found in various fields of modern science, such as astrophysics, condensed matter physics, material science and biomathematics \cite{BC2013KRM,Nicolis,XYZ12,ZRSD}. Under certain regularity and growth assumptions on $f$, the existence and multiplicity of solutions to \eqref{HE1} have been broadly studied over the past decades, see \cite{AR,HL,CKC,Ni,Rab,sqw} and the references therein.

In recent years, various effective numerical algorithms for finding multiple solutions to the semilinear boundary value problem \eqref{HE1} have been developed. Among them, there are a class of methods based on certain minimax principles and optimization algorithms to seek multiple critical points of the corresponding variational energy functional, including the mountain pass algorithm \cite{cys}, the high-linking algorithm \cite{ZHD} and the local minimax method \cite{LZ1,LZ2,XYZ12}, etc.. There are also other methods which discretize directly the boundary value problem and then solve the resulting nonlinear algebraic system for multiple solutions with suitable initial data, such as the search extension method (SEM) \cite{CX2000,CX,CXbook,CX3}, the eigenfunction expansion method (EEM) \cite{ZZY}, the two-level spectral method \cite{WHL2018SISC} and several Newton type methods \cite{XYZ15,FBF2015SISC}. To be more detailed, the SEM was proposed by Chen and Xie, it first searches for initial guesses based on the linear combinations of several eigenfunctions of the Laplacian, then increases the number of eigenfunctions gradually to obtain better initial guesses, and finally solves the discrete problem by a numerical continuation method. To improve the efficiency of SEM, an interpolated coefficient finite element method \cite{XC} for resolving the difficulty caused by the nonlinear term and a two-grid method based on the finite element discretization \cite{LXC2009} are applied. Zhang et al. \cite{ZZY} proposed the EEM, which originates from a similar idea of SEM to approximate the solutions with the linear combination of eigenfunctions of the Laplacian. It is illustrated that, when the number of eigenfunctions is sufficiently large, there exists a unique numerical solution exponentially converging to a specified true solution. Recently, Wang et al. \cite{WHL2018SISC} developed the two-level approach based on spectral methods and homotopy continuation, which first solves the nonlinear problems using spectral discretization with small degree-of-freedom and then solves the corresponding linearized problems by using spectral discretization with large degree-of-freedom.

The traditional SEM \cite{CX,XC,CX3,LXC2009} is based on the finite element strategy (FESEM), which is very convenient to solve for problems in complex geometries. It is well known that spectral methods are popular for solving problems in regular domains (e.g., rectangles, circles, cylinders, or spheres) and can provide superior accuracy so long as the true solution possesses good regularities \cite{Shen2011}. It is easy to see that, under the assumptions
\begin{enumerate}[(A3)]
  \item[${\bf(A_1)}$] $f\in C^k(\mathbb{R},\mathbb{R})$ with suitable $k\geq 4$; and
  \item[${\bf(A_2)}$] there exist constants $C_1,C_2>0$ such that $|f(\xi)|\leq C_1+C_2|\xi|^p$, where $1<p<\infty$,
\end{enumerate}
any weak solution of \eqref{HE1} belongs to $H^{k+2}(\Omega)\cap H_0^1(\Omega)$ \cite{David}. It is natural to design a spectral discretization in order to improve the efficiency and accuracy of SEM in solving \eqref{HE1} for multiple solutions. In fact, as the numerical observations in \cite{CX,XC,LZ1,XYZ12}, the solutions to \eqref{HE1} are often multi-peak and may possess boundary and/or interior layers, especially for those with high Morse indices. Thus the high accuracy of discretization is no doubt very important to effectively simulate these solutions. Li and Wang \cite{LW} proposed an effective pseudo-spectral discretization for computing multiple solutions of the Lane-Emden equation based on the bifurcation method \cite{ZHY}, but the error estimate of this method is still open to be verified. In addition, Wang et al.  \cite{WHL2018SISC} provided the error estimates of their two-level spectral method for a class of one-dimensional semilinear problems.

In this paper, we are aimed to design an interpolated coefficient spectral-Galerkin-type SEM (SGSEM) for finding multiple solutions of \eqref{HE1} and focus on the following two aspects:
\begin{enumerate}[(1)]
  \item Apply the Legendre-Galerkin spectral method to discretize \eqref{HE1} after a reasonable initial guess is obtained by the search and extension steps of the SEM \cite{CX,XC,CX3,LXC2009}. Specifically, by taking the Lagrange basis polynomials associated with the Legendre-Gauss-Lobatto (LGL) points as the basis of the solution space, an interpolated coefficient approach is implemented in the weak form of numerical solutions to obtain a simple and easy-to-solve nonlinear algebraic system. Compared with the FESEM, the SGSEM combining with the interpolated coefficient approach only takes a small degree-of-freedom to discretize the model problem, and leads to a small nonlinear algebraic system due to its spectral convergence as mentioned in the following. As a result, the SGSEM reduces the computational cost dramatically.
  \item Verify the existence, uniqueness and spectral convergence for the SGSEM numerical solution with respect to the corresponding true solution of \eqref{HE1}. It is really challenging to do this due to the nonlinearity of the model problem, and the multiplicity of solutions, the non-coercivity of the bilinear form
  $$
    \langle\mathcal{L}'(u)w,v\rangle=\int_{\Omega}\nabla w\cdot\nabla v\mathrm{d}\x -\int_{\Omega}f'(u)wv\mathrm{d}\x,\quad w,v\in H_0^1(\Omega),
  $$
  for each solution $u$ with Morse index $\mathrm{MI}\geq1$. Moreover, owing to the employment of the interpolated coefficient approach, new difficulties arise in the convergence analysis of the SGSEM. In this paper, the Schauder's fixed point theorem and other technical strategies will be used.
\end{enumerate}

The rest of this paper is organized as follows. In Sect.~\ref{sec:2}, we introduce an interpolated coefficient Legendre-Galerkin spectral discretization and propose the SGSEM to solve \eqref{HE1} for multiple solutions. In Sect.~\ref{sec:3}, we establish the existence, uniqueness and spectral convergence for the SGSEM numerical solution with respect to the corresponding true solution. In Sect.~\ref{sec:4}, numerical results for semilinear problems with cubic and triangular nonlinearities are presented to illustrate the efficiency of our approach and show the properties of multiple solutions. The conclusions are drawn in Sect.~\ref{sec:5}.

\section{An efficient spectral-Galerkin-type search extension method}
\label{sec:2}
As mentioned above, the boundary value problem \eqref{HE1} has multiple solutions when $f$ satisfies certain regularity and growth assumptions. In order to inherit the advantage of the traditional SEM \cite{CX,XC,CX3,LXC2009} in constructing effective initial data for multiple solutions and improve its efficiency, we modify the SEM by discretizing \eqref{HE1} with an efficient spectral-Galerkin method.

Without loss of generality, we assume that $\Omega=(-1,1)^2$ and $f(0)=0$. The weak formulation for \eqref{HE1} reads as follows: find $u\in H_0^1(\Omega)$ such that
\begin{equation}\label{eq:weak}
  \mathcal{A}(u,v)-\big(f(u),v \big)=0, \quad\forall v\in H_0^1(\Omega),
\end{equation}
where $\mathcal{A}(u,v)=(\nabla u,\nabla v)$ and $(\cdot,\cdot)$ is the inner product in $L^2(\Omega)$. Consider the approximation space
$$
  X_N=P_N^0\otimes P_N^0,\quad P_N^0=\{v\in P_N:v(\pm1)=0\},
$$
for some positive integer $N$, where $P_N$ denotes the set of all polynomials defined on $[-1,1]$ with degree less than or equal to $N$. The standard Galerkin approximation for \eqref{HE1} is to find $u_N\in X_N$ such that
\begin{equation}\label{eq:standardGalerkin}
  \mathcal{A}(u_N,v_N)-\big(f(u_N),v_N\big)=0, \quad\forall v_N\in X_N.
\end{equation}
This leads to a nonlinear algebraic system which is often solved by the Newton method.

Due to the presence of nonlinear term, the Jacobian associated with \eqref{eq:standardGalerkin} has to be updated repeatedly during the Newton iterations. Thus, the discretization \eqref{eq:standardGalerkin} usually leads to a very time-consuming and expensive computation for the Newton method. To overcome this drawback, inspired by the interpolated coefficient finite element method (see, e.g., \cite{XC,CLZ1989}) which was first proposed by Zlamal \cite{Zlamal1980} for semilinear parabolic problems, we modify the scheme \eqref{eq:standardGalerkin} by replacing $f(u_N)$ with an appropriate interpolation of it so that the Jacobian can be updated with a very small computational effort. The resulting discretization scheme is called an interpolated coefficient spectral Galerkin method in the present paper.

\subsection{Interpolated coefficient Legendre-Galerkin spectral method}
\label{sec:2.1}
Let $\{h_j\}_{j=0}^N$ be the Lagrange basis polynomials (nodal bases) associated with the Legendre-Gauss-Lobatto (LGL) points $\{\xi_j\}_{j=0}^{N}$ (i.e., $\xi_0=-1$, $\xi_N=1$, and $\{\xi_j\}_{j=1}^{N-1}$ are $N-1$ zeros of the first-order derivative of the Legendre polynomial of degree $N$). Clearly, the two-dimensional tensorized LGL points and the associated nodal bases are respectively given by
\begin{equation}\label{eq:2DLGLp}
  \x_{jk}=(\xi_j,\xi_k),\quad h_{jk}(\x)=h_j(x)h_k(y),\quad
  \x=(x,y)\in\bar{\Omega},\quad j,k=0,1,\cdots,N.
\end{equation}

The proposed interpolated coefficient Legendre-Galerkin spectral method for the boundary value problem \eqref{HE1} is aimed to find $u_N\in X_N$ such that
\begin{equation}\label{eq:Galerkin}
  \mathcal{A}(u_N,v_N)-\big(\mathrm{I}_Nf(u_N),v_N\big)=0, \quad\forall v_N\in X_N,
\end{equation}
where $\mathrm{I}_N:C(\bar{\Omega})\to P_N\otimes P_N$ is the Lagrange interpolation operator such that for all $u\in C(\bar{\Omega})$,
$$
  \big(\mathrm{I}_Nu \big)(\x_{jk})= u(\x_{jk}),\quad j,k=0,1,\cdots,N.
$$
Denote $U_{jk}=u_N(\x_{jk})$ ($j,k=0,1,\cdots,N$), then $U_{j0}=U_{jN}=U_{0k}=U_{Nk}=0$, $j,k=0,1,\cdots,N$. Since $f(0)=0$, we have
$$
  \big(\mathrm{I}_Nf(u_N)\big)(\x)=\sum_{j,k=1}^{N-1}f(U_{jk})h_{jk}(\x).
$$
It is noted that the key idea of the above interpolated coefficient method is interpolating the whole term $f(u_N)$ based on the Lagrange interpolation points. As a result, $f(U_{jk})$  $\left(j,k=0,1,\cdots,N\right)$ turn to be the interpolation coefficients and the nonlinear algebraic system obtained from the weak formulation (\ref{eq:Galerkin}) is much simpler than that of \eqref{eq:standardGalerkin}. Of course, the interpolated coefficient method can also be implemented based on other types of interpolation basis functions, e.g., Fourier interpolation basis functions. For these cases, the interpolation coefficients are generally no longer $f(U_{jk})$, but a simple form of nonlinear algebraic system can still be obtained. The details are omitted here.

Since $u_N\in X_N$, it can be written as
\begin{equation}\label{eq:u-interp}
  u_N(\x)=\sum_{j,k=1}^{N-1}U_{jk}h_{jk}(\x)=\sum_{j,k=1}^{N-1}U_{jk}h_j(x)h_k(y).
\end{equation}
Plugging \eqref{eq:u-interp} into \eqref{eq:Galerkin} and taking $v_N=h_{lm}$ ($l,m=1,2,\cdots,N-1$) leads to  the nonlinear algebraic system
\begin{equation}\label{eq:nonlsys}
  \mathbf{AUB} + \mathbf{BUA} - \mathbf{BF(U)B} = \mathbf{O},
\end{equation}
with
\begin{align*}
  &\mathbf{U}=(U_{jk})_{j,k=1,2,\cdots,N-1},\quad \mathbf{F(U)}=\Big(f(U_{jk})\Big)_{j,k=1,2,\cdots,N-1}, \\
  &\mathbf{A}=(A_{jk})_{j,k=1,2,\cdots,N-1},\quad A_{jk} = \int_{-1}^1h_k'(x)h_j'(x)\mathrm{d}x, \\
  &\mathbf{B}=(B_{jk})_{j,k=1,2,\cdots,N-1},\quad B_{jk} = \int_{-1}^1h_k(x)h_j(x)\mathrm{d}x.
\end{align*}
Eq. \eqref{eq:nonlsys} can also be rewritten in the following form
\begin{equation}\label{eq:nonlsys-tensor}
 \mathbf{K}\mathbf{u} - \mathbf{M}\mathbf{f}\big(\mathbf{u}\big) = \mathbf{0},
\end{equation}
where $\mathbf{K}=\mathbf{B}\otimes\mathbf{A} + \mathbf{A}\otimes\mathbf{B}$, $\mathbf{M}=\mathbf{B}\otimes\mathbf{B}$, $\mathbf{u}$ and $\mathbf{f}(\mathbf{u})$ are vectors of length $(N-1)^2$ formed by the columns of $\mathbf{U}$ and $\mathbf{F(U)}$ respectively, i.e.,
\begin{align}
  \mathbf{u} &= \big(U_{11},U_{21},\cdots,U_{q1},U_{12},\cdots,U_{q2},\cdots,U_{q1},\cdots,U_{qq}\big)^{\top}, \label{eq:vecU}\\
  \mathbf{f}(\mathbf{u}) &= \big(f(U_{11}),f(U_{21}),\cdots,f(U_{q1}), f(U_{12}),\cdots,f(U_{q2}), \cdots, f(U_{q1}),\cdots,f(U_{qq})\big)^{\top},
\end{align}
and $\otimes$ denotes the Kronecker product operator, i.e., $\mathbf{A}\otimes\mathbf{B}=(A_{ij}\mathbf{B})_{i,j=1,2,\cdots,q^2}$ with $q=N-1$. Obviously, the Jacobian corresponding to the nonlinear algebraic system  \eqref{eq:nonlsys-tensor} is
\begin{equation}\label{eq:Jacobian}
  \mathbf{J}(\mathbf{u}) = \mathbf{K} - \mathbf{M}\mathbf{D}_f(\mathbf{u}),
\end{equation}
with
\begin{equation}\label{eq:dfu}
  \mathbf{D}_f(\mathbf{u})=\diag\big(f'(U_{11}),\cdots,f'(U_{q1}), f'(U_{12}),\cdots,f'(U_{q2}), \cdots, f'(U_{q1}),\cdots,f'(U_{qq})\big).
\end{equation}
Thanks to the implementation of the interpolated coefficient method, the expression of the Jacobian (\ref{eq:Jacobian}) is much simple, with $\mathbf{K}$ and $\mathbf{M}$ two constant matrices and only a diagonal matrix $\mathbf{D}_f(\mathbf{u})$ needing to be updated during each Newton iteration. As a matter of fact, it is the advantage of the interpolated coefficient method for discretizing nonlinear problems.

\begin{remark}
By using tensor product, it is straightforward to generalize the above interpolated coefficient Legendre-Galerkin spectral method to higher dimensional cases.
\end{remark}

\subsection{The algorithm for the SGSEM}
\label{sec:2.2}

To ensure the convergence of the algorithm, it is crucial to construct good initial guesses for the solutions of the discrete problem \eqref{eq:Galerkin}. Motivated by the traditional SEM, we propose a spectral-Galerkin-type SEM (SGSEM) based on the discretization \eqref{eq:Galerkin}. In the method, the initial guesses are constructed based on the linear combinations of several eigenfunctions of the corresponding linear eigenvalue problem, and the discrete problem \eqref{eq:Galerkin} is then solved by a numerical continuation method for each initial guess.

Following the original SEM \cite{CX} and its improved algorithms \cite{XC}, we present the outline of the interpolated coefficient SGSEM for finding multiple solutions of \eqref{HE1} below.
\begin{enumerate}[ Step~1.]
  \item Compute eigenpairs of $-\Delta$ in $\Omega$ with the homogeneous Dirichlet boundary conditions.
  \item As much as possible, search for all solutions of the Galerkin approximation of \eqref{HE1} on an appropriately small subspace spanned by some eigenfunctions obtained in Step~1. Choose one solution as a rough initial guess.
  \item Extend the small subspace in Step~2 to a relatively larger subspace to update the initial guess.
  \item Discretize \eqref{HE1} by the interpolated coefficient Legendre-Galerkin spectral method to obtain the nonlinear algebraic system \eqref{eq:nonlsys-tensor}.
  \item Solve \eqref{eq:nonlsys-tensor} by a numerical continuation method based on the Newton approach with the initial guess obtained in Step~2-3.
\end{enumerate}
Notice that Step~2 is a key stage for separating various solutions and determining their rough positions, whereas Step~3 can be omitted if the initial guess constructed in Step~2 is already good enough. Compared with the previously mentioned two-level method based on spectral discretization \cite{WHL2018SISC}, the SGSEM selectively uses the eigenfunctions in Steps~2-3 to obtain initial guesses of multiple solutions according to certain rules. Moreover, in the SGSEM, the interpolated coefficient Legendre-Galerkin spectral method is used to discretize the model problem to obtain a simple nonlinear algebraic system. Therefore, the two methods are very different in the construction of the initial guesses and the spectral discretization.

\section{Existence, uniqueness and spectral convergence for the solution of SGSEM}
\label{sec:3}
In this section, we consider the existence, uniqueness and spectral convergence for the numerical solution of the interpolated coefficient SGSEM with respect to each specified true solution of \eqref{HE1}. Hereafter, we denote any positive constant by $C$ unless it is specified.

\subsection{Preliminaries}
\label{sec:3.1}
First, we introduce some function spaces. For $p\in[1,\infty]$, let
$$
  L^p(\Omega)=\big\{v:\mbox{ $v$ is measurable on $\Omega$ and }\|v\|_{L^p(\Omega)}<\infty\big\},
$$
where
$$
\|v\|_{L^p(\Omega)}=\left(\int_{\Omega}|v(\x)|^p\mathrm{d}\x\right)^{1/p},\quad 1\leq p<\infty,\qquad
\|v\|_{L^\infty(\Omega)}=\mathrm{ess}\sup_{\x\in\Omega}|v(\x)|.
$$
When $p=2$, the space $L^2(\Omega)$ is a Hilbert space equipped with the inner product
$$
(u,v)_{L^2(\Omega)}=\int_{\Omega}u(\x)v(\x)\mathrm{d}\x.
$$
For simplicity, we use $(\cdot,\cdot)$ and $\|\cdot\|_0$ to denote $(\cdot,\cdot)_{L^2(\Omega)}$ and $\|\cdot\|_{L^2(\Omega)}$, respectively. For $m\in\mathbb{N}_+$, set the Sobolev space
$$
H^m(\Omega)=\big\{v\in L^2(\Omega): D^{\alpha}v\in L^2(\Omega), \forall\alpha\in\mathbb{N}^2,|\alpha|\leq m \big \},
$$
equipped with the inner product and norm
$$
(u,v)_m=\int_{\Omega}\sum_{|\alpha|\leq m}D^{\alpha}u(\x)D^{\alpha}v(\x)\mathrm{d}\x,\quad
\|v\|_m=\sqrt{(v,v)_m}.
$$
The space $H_0^m(\Omega)$ is defined as the closure of $C_0^{\infty}(\Omega)$ in $H^m(\Omega)$. The dual space of $H_0^m(\Omega)$ is denoted by $H^{-m}(\Omega)$ with the norm
$$
\|L\|_{-m}=\sup_{0\neq v\in H_0^m(\Omega)}\frac{\langle L,v\rangle}{\|v\|_m}, \quad \quad  \forall L\in H^{-m}(\Omega),
$$
where $\langle\cdot,\cdot\rangle$ is the duality pairing between $H^{-m}(\Omega)$ and $H_0^m(\Omega)$. For any real number $s>0$ that is not an integer, define $H^s(\Omega)=\big\{v\in H^{\lfloor s\rfloor}(\Omega): \|v\|_s<\infty \big\}$ with the norm \cite{Adams1975}
$$
  \|v\|_s = \left(\|v\|_{\lfloor s\rfloor}^2 + \sum_{|\alpha|=\lfloor s\rfloor} \int_{\Omega}\int_{\Omega}\frac{|D^{\alpha}v(\x)-D^{\alpha}v(\y)|^2}{|\x-\y|^{2\beta+2}}\mathrm{d}\x\mathrm{d}\y \right)^{1/2}, \quad  \beta=s-\lfloor s\rfloor.
$$
In addition, the interpolation space of $H^r(\Omega)$ and $H^s(\Omega)$ $(r\geq s\geq0)$ is denoted as $\big[H^r(\Omega),H^s(\Omega)\big]_{\tau}$ ($0<\tau<1$), equipped with the so-called graph norm \cite{Lions-Magenes}. Actually, it is equivalent to a Sobolev space presented in the following lemma (see Theorem 1.5 and Theorem 1.6 in \cite{CB2}), which plays an important role in proving our main theorems.
\begin{lemma}\label{LM-interopation-sp}
For any nonnegative numbers $r$ and $s$ with $r\geq s$, and any $\tau \in(0,1)$, it holds that
\begin{equation}\label{eq:interopation-sp}
\big[H^r(\Omega),H^s(\Omega)\big]_{\tau}=H^{(1-\tau)r+\tau s}(\Omega),
\end{equation}
in the sense of equivalent norms. Further, for all $v\in H^r(\Omega)$, it yields
\begin{equation}\label{eq:interopation-norm}
\|v\|_{[H^r(\Omega),H^s(\Omega)]_{\tau}} \leq \|v\|_r^{1-\tau}\|v\|_{s}^{\tau}.
\end{equation}
\end{lemma}

To discuss the existence, uniqueness and spectral convergence of the numerical solutions corresponding to specified true solutions of \eqref{HE1}, the basic assumptions ${\bf(A_1)}$, ${\bf(A_2)}$, and the following additional assumption are needed:
\begin{enumerate}[(A3)]
	\item[${\bf(A_3)}$] the solution $u\in H^1_0(\Omega)$ of \eqref{HE1} is nonsingular, i.e., $\mathcal{L}'(u):H_{0}^1(\Omega)\rightarrow H^{-1}(\Omega)$ is an isomorphism.
\end{enumerate}
The assumption ${\bf(A_3)}$ is aimed to exclude the limit points and bifurcation points \cite{BR,XZ}, which guarantees that $u$ must be an isolated solution. An application of the open mapping theorem to ${\bf(A_3)}$ yields
\begin{equation}\label{iso}
\|v\|_1\leq C\big\| \mathcal{L}'(u)v\big\|_{-1},\quad\forall v\in H_0^{1}(\Omega).
\end{equation}

For the future development, let's define an orthogonal projection $\mathrm{\Pi}_N^{1,0}: H_0^1(\Omega)\rightarrow {X_N}$ such that
\begin{equation}\label{eq:h1proj}
  \mathcal{A}\big(u-\mathrm{\Pi}_N^{1,0}u, v_N\big)=0, \quad \forall v_N\in {X_N},
\end{equation}
where $\mathcal{A}(u,v)=(\nabla u,\nabla v)$ is obviously a $H_0^1(\Omega)$-coercive bilinear form, i.e., there exists a constant $\gamma^\ast>0$ such that
\begin{equation}\label{eq:A-coercivity}
  \mathcal{A}(v,v)\geq \gamma^\ast\|v\|_1^2,\quad\forall v\in H_0^1(\Omega).
\end{equation}
The basic approximation result is stated as follows (see, e.g., Theorem 3.31 in \cite{xxm}).
\begin{lemma}\label{LM:PI}
For any $u\in H^m(\Omega)\cap H_0^1(\Omega)$ with $m\geq1$, we have
$$
  \big\|u-\mathrm{\Pi}_N^{1,0}u\big\|_r\leq c_1N^{r-m}\|u\|_m,\quad 0\leq r\leq1.
$$
\end{lemma}

In fact, one can construct higher order projection operators due to the following general result (see Theorem 7.4 in \cite{CB2}).
\begin{lemma}\label{LM:PI-highorder}
Let $s\geq0$ such that $s-1/2$ is not an integer. There exists an operator $\mathrm{\Pi}_N^{s,0}$ from $H^s(\Omega)\cap H_0^1(\Omega)$ into $X_N$ such that for any $u\in H^m(\Omega)\cap H_0^1(\Omega)$, $m\geq s$, the following estimate holds
$$
\big\|u-\mathrm{\Pi}_N^{s,0}u\big\|_r\leq CN^{r-m}\|u\|_m,\quad 0\leq r\leq s.
$$
\end{lemma}

Then, for any $u,v\in C(\bar{\Omega})$, define the discrete scalar product and norm
$$
\langle u,v\rangle_N=\sum_{j,k=0}^N u(\x_{jk})v(\x_{jk})\omega_{jk}, \qquad
\|v\|_{N}=\langle v, v\rangle_{N}^{1/2},
$$
where $\x_{jk}$ and $\omega_{jk}$ $(j,k=0,1,\cdots,N)$ are the two-dimensional tensorized LGL quadrature nodes and the corresponding weights, respectively. It is well-known that
$$
\langle u,v\rangle_N=(u,v), \quad \forall uv \in { X_{2N-1}},
$$
and
\begin{equation}\label{eq:norm_0N}
\|v\|_0 \leq \|v\|_{N}\leq c_L \|v\|_0, \quad \forall v\in {X_N},
\end{equation}
with $c_L$ a constant independent of $N$ \cite{CQ,SLW}. By Theorem 14.2 in \cite{CB2}, we have
\begin{lemma}\label{LM:interp}
For all $u\in H^m(\Omega)$ with $m>1+r/2$, it yields
$$
\|u-\mathrm{I}_N u\|_r \leq c_2 N^{r-m}\|u\|_{m},  \quad 0\leq r \leq 1.
$$
\end{lemma}

Further, we list an inverse estimate for polynomials in $X_N$ \cite{BM}.
\begin{lemma} \label{LM-inverse}
For any real number $\mu$ and integer $r$ with $0\leq r\leq \mu$, it is satisfied for any $\varphi\in {X_N}$ that
\begin{equation}\label{inverse}
\|\varphi\|_{\mu}\leq CN^{2(\mu-r)}\|\varphi\|_r.
\end{equation}
\end{lemma}

In order to obtain the existence and error estimates of numerical solutions, a linear auxiliary problem as follows is introduced, i.e., find $w\in H_0^1(\Omega)$ such that
\begin{equation}\label{eq:auxiliary}
  \mathcal{B}(u;w,v):= \mathcal{A}(w,v)-\big(f'(u)w,v\big)=(g,v), \quad\forall v\in H_0^1(\Omega),
\end{equation}
with $g\in L^2(\Omega)$. By using Sobolev imbedding theorems, one can verify that if $u\in H^m(\Omega)$ with $m>1$ and the hypothesis ${\bf(A_1)}$ is satisfied, the bilinear form $\mathcal{B}(u;\cdot,\cdot)$ is bounded, i.e., 
\begin{equation}\label{eq:Bu-boundedness}
  \mathcal{B}(u;w,v) \leq \left(1+\|f'(u)\|_{L^\infty(\Omega)}\right)\|w\|_1\|v\|_1,\quad\forall w,v\in H^1(\Omega).
\end{equation}
Thanks to the results in \cite{CX3}, we have the conclusion below.

\begin{lemma}\label{LM-regularity}
Assume that ${\bf(A_1)}$, ${\bf(A_2)}$ and ${\bf(A_3)}$ hold and $u\in H^{m}(\Omega)\cap H_0^1(\Omega)$ ($m>1$) is a solution of \eqref{HE1}. Then, in any subspace $S$ of $H_{0}^1(\Omega)$, the homogeneous linear problem
\begin{equation}\label{eq:homogenous}
 \mathcal{B}(u;w,v)=0, \quad\forall v\in S,
 \end{equation}
has only zero solution $w=0$. Moreover, there exists a unique solution
$w \in H_{0}^1(\Omega)$ of \eqref{eq:auxiliary} such that
\begin{equation}\label{eq:regularity}
 \|w\|_{2}\leq C \|g\|_0,
\end{equation}
with $C$ a constant only dependent on $u$.
\end{lemma}
\begin{proof}
The assumption ${\bf(A_3)}$ directly implies that \eqref{eq:homogenous} possesses only zero solution. Then, similar to the proofs of Theorems 1 and 2 in \cite{CX3}, the optimal regularity estimate (\ref{eq:regularity}) can be verified by virtue of the compact imbedding and contradiction argument.
\end{proof}

\subsection{Existence and spectral convergence for numerical solutions}
\label{sec:3.2}
Define an elliptic projection $\mathrm{R}_N: H_{0}^1(\Omega)\rightarrow {X_N}$ with respect to the operator $\mathcal{B}(u;\cdot,\cdot)$, such that for any $w\in H_0^1(\Omega)$,
\begin{equation}\label{eq:projection2}
\mathcal{B}\big(u;w-\mathrm{R}_Nw,v_N\big)=0, \quad \forall v_N\in {X_N}.
\end{equation}
The existence and uniqueness of the elliptic projection $\mathrm{R}_N$ is given in the following theorem.

\begin{theorem}\label{TH3.1}
Assume that ${\bf(A_1)}$, ${\bf(A_2)}$ and ${\bf(A_3)}$ hold and $u\in H^{m}(\Omega)\cap H_0^1(\Omega)$ ($m>1$) is a solution of \eqref{HE1}. Then, for any $ w\in H_{0}^1(\Omega)$, there exists a unique $\mathrm{R}_Nw\in {X_N}$ satisfying \eqref{eq:projection2}.
\end{theorem}
\begin{proof}
Because the existence is equivalent to the uniqueness in the finite-dimensional linear problem, we only need to prove the uniqueness. For any $w\in H_{0}^1(\Omega)$, suppose that there exist $\mathrm{R}_N^1w$ and  $\mathrm{R}_N^2w$ in $X_N$ satisfying (\ref{eq:projection2}). Then
\begin{equation}\label{eq:unique}
\mathcal{B}\big(u;\mathrm{R}_N^1w-\mathrm{R}_N^2w,v_N\big)=0,\quad \forall v_N\in X_N.
\end{equation}
Since ${X_N}$ is a subspace of $H_{0}^1(\Omega)$, by Lemma~\ref{LM-regularity}, \eqref{eq:unique} implies that $\mathrm{R}_N^1w=\mathrm{R}_N^2w$. The uniqueness and then Theorem~\ref{TH3.1} are proved.
\end{proof}

\begin{theorem}\label{TH3.2}
Assume that ${\bf(A_1)}$, ${\bf(A_2)}$ and ${\bf(A_3)}$ hold and $u\in H^{m}(\Omega)\cap H_0^1(\Omega)$ ($m>1$) is a solution of \eqref{HE1}. Then, for all $w\in H^{m}(\Omega)\cap H_0^1(\Omega)$, the following estimates holds
$$
  \|w-\mathrm{R}_Nw\|_1\leq CN^{1-m}\|w\|_m, \quad
  \|w-\mathrm{R}_Nw\|_0\leq CN^{-m}\|w\|_m,
$$
for $N$ large enough. As a result, there exist two constants $c_3$ and $c_4$ dependent only on $u$ such that
\begin{equation}\label{h529}
\|u-\mathrm{R}_Nu\|_{1}\leq c_3N^{1-m}, \quad
\|u-\mathrm{R}_Nu\|_0\leq c_4N^{-m},
\end{equation}
for $N$ large enough.
\end{theorem}
\begin{proof}
Denote $\tilde{e}=w-\mathrm{R}_Nw$. Let $\tilde{w}\in H_{0}^1(\Omega)$ be the solution of the auxiliary linear problem (\ref{eq:auxiliary}) with $g=\tilde{e}$. In terms of \eqref{eq:projection2}, \eqref{eq:Bu-boundedness}, Lemma~\ref{LM:PI}, and Lemma~\ref{LM-regularity}, we have
\begin{align*}
(\tilde{e},\tilde{e})
 &= \mathcal{B}(u;\tilde{e},\tilde{w})
  = \mathcal{B}\big(u;\tilde{e},\tilde{w}-\mathrm{\Pi}_N^{1,0}\tilde{w}\big) \\
 &\leq C\|\tilde{e}\|_1 \big\|\tilde{w}-\mathrm{\Pi}_N^{1,0}\tilde{w}\big\|_1
 \leq CN^{-1}\|\tilde{e}\|_1\|\tilde{w}\|_2
 \leq CN^{-1}\|\tilde{e}\|_1\|\tilde{e}\|_0.
\end{align*}
Thus,
\begin{equation}\label{h530}
\|\tilde{e}\|_0\leq CN^{-1}\|\tilde{e}\|_{1}.
\end{equation}
On the other hand, noting that $u\in H^m(\Omega)\hookrightarrow C(\bar{\Omega})$, by virtue of ${\bf(A_1)}$, \eqref{eq:A-coercivity}, \eqref{eq:Bu-boundedness} and \eqref{h530}, we obtain
\begin{align*}
\gamma^\ast\|\tilde{e}\|_1^2
  \leq \mathcal{A}(\tilde{e},\tilde{e})
 &=\mathcal{B}(u;\tilde{e},\tilde{e}) + \big(f'(u)\tilde{e},\tilde{e}\big) \\
 &=\mathcal{B}\big(u;\tilde{e},w-\mathrm{\Pi}_N^{1,0}w\big) + \big(f'(u)\tilde{e},\tilde{e}\big) \\
 &\leq C\|\tilde{e}\|_1 \big\|w-\mathrm{\Pi}_N^{1,0}w\big\|_1 + \big\|f'(u)\big\|_{L^{\infty}(\Omega)}\|\tilde{e}\|_0^2 \\
 &\leq CN^{1-m}\|\tilde{e}\|_1\|w\|_m + CN^{-2}\|\tilde{e}\|_1^2,
\end{align*}
which implies the estimate of $\|\tilde{e}\|_1$. Then the estimate of $\|\tilde{e}\|_0$ follows from (\ref{h530}).
\end{proof}

Recall that the weak solution $u\in H^1_0(\Omega)$ of problem \eqref{HE1} satisfies
\begin{equation}\label{eq:HE2-weak}
 \mathcal{A}(u,v)-\big(f(u),v\big)=0, \quad \forall v\in H^1_0(\Omega).
\end{equation}
On the other hand, the numerical solution $u_N\in{X_N}$ for \eqref{HE1} with the interpolated coefficient Legendre-Galerkin spectral method is determined by
\begin{equation}\label{eq:HE3-weak}
\mathcal{A}(u_N,v_N)-\big(\mathrm{I}_Nf(u_N),v_N\big)=0,\quad  \forall v_N \in {X_N}.
\end{equation}
Denote $e=u-u_N$. According to \eqref{eq:HE2-weak} and \eqref{eq:HE3-weak}, we have
\begin{align}
\mathcal{B}(u;e,v_N)
 &= \mathcal{A}(u,v_N)-\mathcal{A}(u_N,v_N)-\big(f'(u)(u-u_N),v_N\big)  \nonumber\\
 &= \Big(f(u)-\mathrm{I}_Nf(u_N)-f'(u)(u-u_N),v_N\Big), \quad \forall v_N\in X_N. \label{eq:error-equation}
\end{align}
In order to prove one of our main theorems, we need the following technical result.

\begin{lemma}\label{LM:Ruxv}
Assume that ${\bf(A_1)}$, ${\bf(A_2)}$ and ${\bf(A_3)}$ hold and $u\in H^{m}(\Omega)\cap H_0^1(\Omega)$ ($k\geq m>3$) is a solution of \eqref{HE1}. Denote
\begin{equation}\label{eq:Ruxv-def}
\mathcal{R}(u,\chi,v_N)=\Big(f(u)-\mathrm{I}_Nf(\chi)-f'(u)(u-\chi),v_N\Big),\quad\forall \chi, v_N \in X_N.
\end{equation}
If $\chi\in X_N$ satisfies $\max\{\|\chi\|_{L^{\infty}(\Omega)},\|u\|_{L^{\infty}(\Omega)}\}\leq\delta$ for some constant $\delta>0$, we have
\begin{equation}\label{eq:Ruxv}
\big|\mathcal{R}(u,\chi,v_N)\big|
\leq \Big(c_5N^{-m} + CN^{2\tau}\|\mathrm{I}_Nu-\chi\|_1\|\mathrm{I}_Nu-\chi\|_0 + CN^{\tau-1}\|\mathrm{R}_Nu-\chi\|_1\Big)\|v_N\|_0,
\end{equation}
where $\tau\in(0,1)$ is an arbitrary parameter and
$$
  c_5 = c_2\|f(u)\|_m + c_{\delta}\big(c_L(c_2\|u\|_m+c_4)+c_4\big),
$$
with $c_L$, $c_2$, $c_4$ the constants in \eqref{eq:norm_0N}, Lemma~\ref{LM:interp}, Theorem~\ref{TH3.2}, respectively, and $c_\delta$ the maximum of $|f'|,|f''|,|f'''|$ on $[-\delta,\delta]$.
\end{lemma}
\begin{proof}
Obviously,
\begin{align}
\big|\mathcal{R}(u,\chi,v_N)\big|
  &\leq \big\|f(u)-\mathrm{I}_Nf(\chi)-f'(u)(u-\chi)\big\|_0 \|v_N\|_0 \nonumber\\
  &\leq \big(\gamma_1(u,\chi)+\gamma_2(u,\chi)+\gamma_3(u,\chi)\big)\|v_N\|_0,
\label{eq:absR}
\end{align}
with
\begin{align*}
 \gamma_1(u,\chi) &= \big\|f(u)-\mathrm{I}_Nf(u)\big\|_0, \\
 \gamma_2(u,\chi) &= \big\|\mathrm{I}_N\big(f(u)-f(\chi)-f'(u)(u-\chi)\big)\big\|_0, \\
 \gamma_3(u,\chi) &= \big\|\mathrm{I}_N\big(f'(u)(u-\chi)\big)-f'(u)(u-\chi)\big\|_0.
\end{align*}
By applying the fact $u\in H^m(\Omega)\hookrightarrow C(\bar{\Omega})$, condition ${\bf(A_1)}$ with $k\geq m$, and Lemma~\ref{LM:interp}, we have
\begin{equation}\label{eq:gam1}
  \gamma_1(u,\chi) \leq c_2\|f(u)\|_mN^{-m}.
\end{equation}
Denote $\phi(t)=f\big(\chi+t(u-\chi)\big)$. Integration by parts yields
$$
\phi(1)-\phi(0)=\int_0^1\phi'(t)dt=\phi'(1)-\int_0^1\phi''(t)tdt.
$$
In terms of the expression of  $\phi(t)$, we have
$$
f(u)-f(\chi)-f'(u)(u-\chi)=-\left(\int_0^1 f''\big(\chi+t(u-\chi)\big)tdt\right)(u-\chi)^2.
$$
Then, in terms of assumption ${\bf (A_1)}$ and $\|\chi+t(u-\chi)\|_{L^{\infty}(\Omega)}\leq\delta\; (\forall t\in[0,1])$, we have
\begin{align}
\gamma_2(u,\chi)
 &\leq \big\| f(u)-f(\chi)-f'(u)(u-\chi) \big\|_N \nonumber\\
 &\leq \left\|\int_0^1 f''\big(\chi+t(u-\chi)\big)tdt\right\|_{L^{\infty}(\Omega)} \big\|(u-\chi)^2\big\|_N \nonumber\\
 &\leq c_{\delta}\big\|(u-\chi)^2\big\|_N \nonumber\\
 &= c_{\delta}\big\|(\mathrm{I}_Nu-\chi)^2\big\|_N \nonumber\\
 &\leq c_{\delta}\big\|\mathrm{I}_Nu-\chi\big\|_{L^{\infty}(\Omega)}\big\|\mathrm{I}_Nu-\chi\big\|_N \nonumber\\
 &\leq C\big\|\mathrm{I}_Nu-\chi\big\|_{1+\tau}\big\|\mathrm{I}_Nu-\chi\big\|_0 \nonumber\\
 &\leq CN^{2\tau}\big\|\mathrm{I}_Nu-\chi\big\|_1\big\|\mathrm{I}_Nu-\chi\big\|_0,
 \label{eq:gam2}
\end{align}
where \eqref{eq:norm_0N}, the Sobolev imbedding $H^{1+\tau}(\Omega)\hookrightarrow C(\bar{\Omega})$ $(\tau>0)$, and the inverse estimate in Lemma~\ref{LM-inverse} are used.

Afterwards, one can see that
\begin{align*}
\gamma_3(u,\chi)
 &= \big\|\mathrm{I}_N\big(f'(u)(u-\chi)\big)-f'(u)(u-\chi)\big\|_0 \\
 &\leq \big\|\mathrm{I}_N\big(f'(u)(u-\mathrm{R}_Nu)\big)\big\|_0 + \big\|f'(u)(u-\mathrm{R}_Nu)\big\|_0 \\
 &\quad\; + \big\|\mathrm{I}_N\big(f'(u)(\mathrm{R}_Nu-\chi)\big)-f'(u)(\mathrm{R}_Nu-\chi)\big\|_0.
\end{align*}
From \eqref{eq:norm_0N} and Theorem~\ref{TH3.2}, we obtain
\begin{align*}
\big\|\mathrm{I}_N\big(f'(u)(u-\mathrm{R}_Nu)\big)\big\|_0
 &\leq \big\|\mathrm{I}_N\big(f'(u)(u-\mathrm{R}_Nu)\big)\big\|_N \\
 &= \big\|f'(u)(u-\mathrm{R}_Nu)\big\|_N \\
 &\leq c_{\delta}\big\|u-\mathrm{R}_Nu\big\|_N \\
 &= c_{\delta}\big\|\mathrm{I}_Nu-\mathrm{R}_Nu\big\|_N \\
 &\leq c_{\delta}c_L\big\|\mathrm{I}_Nu-\mathrm{R}_Nu\big\|_0 \\
 &\leq c_{\delta}c_L\Big(\big\|\mathrm{I}_Nu-u\big\|_0+\big\|u-\mathrm{R}_Nu\big\|_0\Big) \\
 &\leq c_{\delta}c_L(c_2\|u\|_m+c_4)N^{-m}.
\end{align*}
Moreover,
$$
\big\|f'(u)(u-\mathrm{R}_Nu)\big\|_0
  \leq c_{\delta}\big\|u-\mathrm{R}_Nu\big\|_0
  \leq c_{\delta}c_4N^{-m}.
$$
On the other hand, Lemma~\ref{LM-interopation-sp}, Lemma~\ref{LM:interp}, Lemma~\ref{LM-inverse} and ${\bf(A_1)}$ imply that
\begin{align}
 &\quad\; \big\|\mathrm{I}_N\big(f'(u)(\mathrm{R}_Nu-\chi)\big)-f'(u)(\mathrm{R}_Nu-\chi)\big\|_0 \nonumber\\
 &\leq c_2N^{-(1+\tau)}\big\|f'(u)(\mathrm{R}_Nu-\chi)\big\|_{1+\tau} \nonumber\\
 &\leq CN^{-(1+\tau)}\big\|f'(u)(\mathrm{R}_Nu-\chi)\big\|_{[H^2(\Omega),H^1(\Omega)]_{1-\tau}} \nonumber\\
 &\leq CN^{-(1+\tau)}\big\|f'(u)(\mathrm{R}_Nu-\chi)\big\|_2^{\tau}\big\|f'(u)(\mathrm{R}_Nu-\chi)\big\|_1^{1-\tau} \nonumber\\
 &\leq CN^{-(1+\tau)}\big\|\mathrm{R}_Nu-\chi\big\|_2^{\tau}\big\|\mathrm{R}_Nu-\chi\big\|_1^{1-\tau} \nonumber\\
 &\leq C N^{\tau-1}\big\|\mathrm{R}_Nu-\chi\big\|_1, \label{eq:gam30}
\end{align}
where the facts that $H^{1+\tau}(\Omega)=\left[H^2(\Omega),H^1(\Omega)\right]_{1-\tau}$ $(0<\tau<1)$ and $u\in H^m(\Omega)\hookrightarrow C^2(\bar{\Omega})\; (m>3)$ are used. Consequently,
\begin{equation}\label{eq:gam3}
\gamma_3(u,\chi)
  \leq c_{\delta}\big(c_L(c_2\|u\|_m+c_4)+c_4\big)N^{-m} + C N^{\tau-1}\big\|\mathrm{R}_Nu-\chi\big\|_1.
\end{equation}
Combining \eqref{eq:absR}-\eqref{eq:gam2} and \eqref{eq:gam3} implies \eqref{eq:Ruxv} and completes the proof.
\end{proof}

Then, we give the first main result. 

\begin{theorem}\label{TH3.3}
Assume that ${\bf(A_1)}$, ${\bf(A_2)}$ and ${\bf(A_3)}$ hold and $u\in H^{m}(\Omega)\cap H_0^1(\Omega)$ ($k\geq m>3$) is a solution of \eqref{HE1}. Then, for sufficiently large $N$, there exists an $u_N\in{X_N}$ satisfying (\ref{eq:HE3-weak}) and
\begin{equation}\label{h531}
\|u-u_N\|_{1}\leq CN^{1-m},\quad \|u-u_N\|_0\leq CN^{-m}.
\end{equation}
\end{theorem}

\begin{proof}
Denote $e=u-u_N$, which can be decomposed into $e=e^\ast+\theta$ with $e^\ast=u-\mathrm{R}_Nu$ and $\theta=\mathrm{R}_Nu-u_N$. According to \eqref{eq:error-equation} and \eqref{eq:Ruxv-def}, one concludes that $u_N\in X_N$ satisfies \eqref{eq:HE3-weak} if and only if $u_N$ satisfies
\begin{equation}\label{eq:equivalent}
\mathcal{B}\big(u;u_N,v_N\big)=\mathcal{B}(u;u,v_N)-\mathcal{R}(u,u_N,v_N),\quad \forall v_N \in X_N.
\end{equation}
In the following, we prove the existence of the numerical solution $u_N$ determined by (\ref{eq:HE3-weak}) and obtain the estimates for $\|\theta\|_1$ and  $\|\theta\|_0$ corresponding to such an $u_N$.

Inspired by the idea of \cite{XZ}, we introduce a mapping $\Phi:X_N\rightarrow X_N$ defined by
\begin{equation}\label{eq:Phi}
\mathcal{B}(u;\Phi(\chi),v_N)=\mathcal{B}(u;u,v_N)-\mathcal{R}(u,\chi,v_N),\quad \forall v_N \in X_N.
\end{equation}
From Lemma~\ref{LM-regularity} and the fact that the existence is equivalent to the uniqueness of the solution to the finite-dimensional linear problem, one can see that for any $\chi\in X_N$, there exists a unique $\Phi(\chi)\in X_N$ satisfying the above definition. Then (\ref{eq:equivalent}) and (\ref{eq:Phi}) imply that any fixed point of $\Phi$ in ${X_N}$ is a solution of (\ref{eq:HE3-weak}). Actually, we will prove the existence of fixed points of $\Phi$ on the convex compact subset $Q$ of $X_N$, which is defined as
$$
  Q=\left\{\chi\in {X_N}\  \big| \ \|\chi-\mathrm{R}_Nu\|_{1}\leq c_6N^{1/2-m} \;\mbox{and}\; \|\chi-\mathrm{R}_Nu\|_0\leq c_7N^{-m} \right\}
$$
for a suitably large but fixed $N$. Here $c_6>0$ is a fixed constant and $c_7=2c_0c_5$ with $c_5$ the constant in Lemma~\ref{LM:Ruxv} and $c_0$ the constant in the estimate
\begin{equation}\label{eq:wc0g}
\big\|\mathrm{\Pi}_N^{1,0}w\big\|_0=\big\|\mathrm{\Pi}_N^{1,0}w-w+w\big\|_0\leq CN^{-1}\|w\|_1+\|w\|_0 \leq c_0\|g\|_0
\end{equation}
for the solution $w$ of the auxiliary linear problem (\ref{eq:auxiliary}). To do this, the Schauder's fixed point theorem will be used. We now verify that $\Phi(Q)\subset Q$ and $\Phi$ is continuous for $N$ large enough.

Let $w\in H^1_0(\Omega)$ be the solution of the auxiliary linear problem (\ref{eq:auxiliary}). Then, for any $\chi \in Q$ and $v_N\in X_N$, it holds
\begin{align}
\big(g,\Phi(\chi)-\mathrm{R}_Nu \big)
&= \mathcal{B} \big(u;\Phi(\chi)-\mathrm{R}_Nu,w \big) \nonumber\\
&= \mathcal{B} \big(u;\Phi(\chi)-\mathrm{R}_Nu,w-v_N \big)+\mathcal{B} \big(u;\Phi(\chi)-\mathrm{R}_Nu,v_N \big) \nonumber\\
&= \mathcal{B} \big(u;\Phi(\chi)-\mathrm{R}_Nu,w-v_N \big)-\mathcal{R}(u,\chi,v_N) \nonumber\\
&\leq C \big\|\Phi(\chi)-\mathrm{R}_Nu \big\|_{1}\|w-v_{N}\|_{1} + \big|\mathcal{R}(u,\chi,v_N)\big|,
\label{eq:gPhi-RNu}
\end{align}
where the definition of $\Phi$ and (\ref{eq:projection2}) are implemented. Noting that, according to Lemma~\ref{LM:PI-highorder}, Lemma~\ref{LM-inverse} and Theorem~\ref{TH3.2}, for all $\chi\in Q$, $1<s<3/2$ and $N$ large enough, we obtain
\begin{align}
\big\|\chi-u\big\|_{L^{\infty}(\Omega)}
 &\leq C\big\|\chi-u\big\|_s
  \leq C\big\|\chi-\mathrm{\Pi}_N^{s,0}u\big\|_s + C\big\|\mathrm{\Pi}_N^{s,0}u-u\big\|_s \nonumber\\
 &\leq CN^{2s-2}\Big(\big\|\chi-\mathrm{R}_Nu\big\|_1+\big\|\mathrm{R}_Nu-u\big\|_1+\big\|u-\mathrm{\Pi}_N^{s,0}u\big\|_1\Big) + CN^{s-m}\|u\|_m \nonumber\\
 &\leq CN^{2s-2}N^{s-m}+ CN^{s-m}\|u\|_m= \mathrm{o}(1).
 \label{eq:chi-bound}
\end{align}
Then, for suitably large $N$, $\|\chi\|_{L^{\infty}(\Omega)}\leq 2\|u\|_{L^{\infty}(\Omega)}$. Lemma~\ref{LM:interp} and Theorem~\ref{TH3.2} imply that
$$
\big\|\mathrm{I}_Nu-\chi\big\|_r
\leq \big\|\mathrm{I}_Nu-u\big\|_r+\big\|u-\mathrm{R}_Nu\big\|_r+\big\|\mathrm{R}_Nu-\chi\big\|_r
\leq CN^{r-m},\quad r=0,1.
$$
Thus, Lemma~\ref{LM:Ruxv} states that for $0<\tau<1/2$,
\begin{align}
\big|\mathcal{R}(u,\chi,v_N)\big|
 &\leq \Big(c_5N^{-m}+CN^{2\tau}N^{1-m}N^{-m}+CN^{\tau-1}N^{1/2-m}\Big) \|v_N\|_0 \nonumber\\
 &\leq \Big(c_5N^{-m}+\mathrm{o}\left(N^{-m}\right)\Big) \|v_N\|_0,  \quad \forall v_N \in {X_N}.
\label{eq:gamma}
\end{align}
Combining \eqref{eq:gPhi-RNu} and \eqref{eq:gamma}, and taking $v_{N}=\mathrm{\Pi}_N^{1,0}w$, we obtain
\begin{align*}
\big(g,\Phi(\chi)-\mathrm{R}_Nu\big)
 &\leq C\big\|\Phi(\chi)-\mathrm{R}_Nu\big\|_1 \big\|w-\mathrm{\Pi}_N^{1,0}w\big\|_1 + \Big(c_5N^{-m}+\mathrm{o}\left(N^{-m}\right)\Big) \big\|\mathrm{\Pi}_N^{1,0}w\big\|_0 \\
 &\leq CN^{-1}\big\|\Phi(\chi)-\mathrm{R}_Nu\big\|_1\|w\|_2 + \Big(c_0c_5N^{-m}+\mathrm{o}\left(N^{-m}\right)\Big)\|g\|_0 \\
 &\leq CN^{-1}\big\|\Phi(\chi)-\mathrm{R}_Nu\big\|_1\|g\|_0 + \left(\frac{c_7}{2}N^{-m}+\mathrm{o}\left(N^{-m}\right)\right)\|g\|_0,
\end{align*}
where Lemma~\ref{LM:PI}, \eqref{eq:wc0g} and \eqref{eq:regularity} are applied. Taking $g=\Phi(\chi)-\mathrm{R}_Nu$, it is easy to obtain
\begin{equation}\label{h535}
\big\|\Phi(\chi)-\mathrm{R}_Nu \big\|_0
  \leq \frac{2}{3}c_7N^{-m} + CN^{-1}\big\|\Phi(\chi)-\mathrm{R}_Nu \big\|_{1},
\end{equation}
for $N$ large enough. On the other hand, by using \eqref{eq:A-coercivity}, \eqref{eq:projection2}, \eqref{eq:Phi} and \eqref{eq:gamma}, we have
\begin{align*}
\gamma^\ast\|\Phi(\chi)-\mathrm{R}_Nu\|_{1}^2
 &\leq \mathcal{A}\big(\Phi(\chi)-\mathrm{R}_Nu,\Phi(\chi)-\mathrm{R}_Nu\big)\\
 &=\mathcal{B}\big(u;\Phi(\chi)-\mathrm{R}_Nu,\Phi(\chi)-\mathrm{R}_Nu\big) \\
 &\quad\; + \Big(f'(u)\big(\Phi(\chi)-\mathrm{R}_Nu\big),\Phi(\chi)-\mathrm{R}_Nu\Big)\\
 &\leq \big|\mathcal{R}\big(u,\chi,\Phi(\chi)-\mathrm{R}_Nu\big)\big| + C \big\|\Phi(\chi)-\mathrm{R}_Nu\big\|_0^2 \\
 &\leq CN^{-m}\big\|\Phi(\chi)-\mathrm{R}_Nu\big\|_0 + C \big\|\Phi(\chi)-\mathrm{R}_Nu\big\|_0^2,
\end{align*}
which implies
\begin{equation}\label{h536}
\big\|\Phi(\chi)-\mathrm{R}_Nu\big\|_1 \leq CN^{-m} + C\big\|\Phi(\chi)-\mathrm{R}_Nu\big\|_0.
\end{equation}
Combining (\ref{h535}) and (\ref{h536}), we obtain
$$
\big\|\Phi(\chi)-\mathrm{R}_Nu\big\|_1 \leq CN^{-m}\leq c_6N^{1/2-m},\quad
\big\|\Phi(\chi)-\mathrm{R}_Nu\big\|_0\leq c_7N^{-m},
$$
for $N$ large enough. This shows $\Phi(Q)\subset Q$.

To prove $\Phi$ is continuous on $Q$, for $\chi_1,\chi_2\in Q$, we have
\begin{align*}
\big(g,\Phi(\chi_1)-\Phi(\chi_2)\big)
 &= \mathcal{B}\big(u;\Phi(\chi_1)-\Phi(\chi_2),w \big)\\
 &= \mathcal{B}\big(u;\Phi(\chi_1)-\Phi(\chi_2),w-v_N\big) + \mathcal{B}\big(u;\Phi(\chi_1)-\Phi(\chi_2),v_N\big) \\
 &\leq C\big\|\Phi(\chi_1)-\Phi(\chi_2)\big\|_1\|w-v_N\|_1 + \big|\mathcal{B}\big(u;\Phi(\chi_1)-\Phi(\chi_2),v_N\big)\big|,
\end{align*}
for all $v_N\in X_N$. Since $u\in H^m(\Omega)\hookrightarrow C(\bar{\Omega})$, $\|\chi_i\|_{L^{\infty}(\Omega)}\leq 2\|u\|_{L^{\infty}(\Omega)}\;(i=1,2)$ from \eqref{eq:chi-bound} for $N$ large enough, we further have
\begin{align}
\big|\mathcal{B}\big(u;\Phi(\chi_1)-\Phi(\chi_2),v_N\big)\big|
 &= \big|\mathcal{R}(u,\chi_2,v_N)-\mathcal{R}(u,\chi_1,v_N)\big| \nonumber\\
 &= \Big| \Big(\mathrm{I}_Nf(\chi_1)-\mathrm{I}_Nf(\chi_2)-f'(u)(\chi_1-\chi_2),v_{N}\Big) \Big| \nonumber\\
 &\leq \big\|\mathrm{I}_N\big(f(\chi_1)-f(\chi_2)\big)\big\|_0\|v_N\|_0 + \big\|f'(u)(\chi_1-\chi_2)\big\|_0\|v_N\|_0 \nonumber\\
 &\leq \big\|f(\chi_1)-f(\chi_2)\big\|_N\|v_N\|_0 + \big\|f'(u)\|_{L^{\infty}(\Omega)}\|\chi_1-\chi_2\|_0\|v_N\|_0 \nonumber\\
 &\leq \left\|\int_0^1f'\big(\chi_2+t(\chi_1-\chi_2)\big)dt\right\|_{L^{\infty}(\Omega)}\|\chi_1-\chi_2\|_N\|v_N\|_0 \nonumber\\
 &\quad\; + \big\|f'(u)\|_{L^{\infty}(\Omega)}\|\chi_1-\chi_2\|_0\|v_N\|_0 \nonumber\\
 &\leq C\|\chi_1-\chi_2\|_0\|v_N\|_0,
 \label{eq:continous}
\end{align}
where \eqref{eq:Phi}, \eqref{eq:norm_0N} and ${\bf(A_1)}$ are used. Consequently,
$$
\big(g,\Phi(\chi_1)-\Phi(\chi_2)\big)
  \leq C\big\|\Phi(\chi_1)-\Phi(\chi_2)\big\|_1\|w-v_N\|_1 + C\|\chi_1-\chi_2\|_0\|v_N\|_0.
$$
Taking $v_N=\mathrm{\Pi}_N^{1,0}w$ and $g=\Phi(\chi_1)-\Phi(\chi_2)$, similar to the process of getting (\ref{h535}),  we obtain
\begin{equation}\label{h538}
\big\|\Phi(\chi_1)-\Phi(\chi_2)\big\|_0
  \leq CN^{-1} \big\|\Phi(\chi_1)-\Phi(\chi_2)\big\|_1 + C\|\chi_1-\chi_2\|_0.
\end{equation}
On the other hand, by virtue of \eqref{eq:A-coercivity}, \eqref{eq:continous}, ${\bf(A_1)}$ and $u\in H^m(\Omega)\hookrightarrow C(\bar{\Omega})$, one obtains
\begin{align*}
\gamma^\ast\|\Phi(\chi_1)-\Phi(\chi_2)\|_1^2
 &\leq \mathcal{A}\big(\Phi(\chi_1)-\Phi(\chi_2),\Phi(\chi_1)-\Phi(\chi_2)\big) \\
 &= \mathcal{B}\big(u;\Phi(\chi_1)-\Phi(\chi_2),\Phi(\chi_1)-\Phi(\chi_2) \big) \\
 & \quad\; +\big(f'(u)\big(\Phi(\chi_1)-\Phi(\chi_2) \big),\Phi(\chi_1)-\Phi(\chi_2) \big) \\
 &\leq C\|\chi_1-\chi_2\|_0 \big\|\Phi(\chi_1)-\Phi(\chi_2) \big\|_0
   + C\big\|\Phi(\chi_1)-\Phi(\chi_2) \big\|_0^2,
\end{align*}
which leads to
\begin{equation}\label{h539}
\big\|\Phi(\chi_1)-\Phi(\chi_2)\big\|_1
  \leq C\|\chi_1-\chi_2\|_0 + C\big\|\Phi(\chi_1)-\Phi(\chi_2) \big\|_0.
\end{equation}
Combining \eqref{h538} and \eqref{h539}, we have, for sufficiently large $N$,
$$
\big\|\Phi(\chi_1)-\Phi(\chi_2) \big\|_1
  \leq C\|\chi_1-\chi_2\|_0
  \leq C\|\chi_1-\chi_2\|_1.
$$
This implies that $\Phi$ is continuous on $Q$.

The implementation of Schauder's fixed point theorem shows the existence of an $u_N\in Q \subset{X_N}$ so that $\Phi(u_N)=u_N$ is a solution of \eqref{eq:HE3-weak}. For such an $u_N$, we have
$$
\|\theta\|_{1}\leq c_6N^{1-m}, \quad\quad  \|\theta\|_0\leq c_7N^{-m},
$$
with $\theta=\mathrm{R}_Nu-u_N$. Further,
$$
\|e\|_1\leq \|e^\ast\|_1+\|\theta\|_{1}\leq CN^{1-m}, \quad\quad  \|e\|_0\leq  \|e^\ast\|_0+\|\theta\|_0\leq CN^{-m}.
$$
Theorem~\ref{TH3.3} is proved.
\end{proof}

\begin{remark}
The assumption $k\geq m>3$ in Theorem~\ref{TH3.3} is due to the derivations of \eqref{eq:gam1} and \eqref{eq:gam30}. When the smoothness of $f$ and the regularity of the solution are more weak, it remains to be proved whether the error estimation \eqref{h531} holds. But for the SGSEM based on the classical spectral-Galerkin discretization, the decomposition of $\mathcal{R}(u,\chi,v_N)$ in Lemma~\ref{LM:Ruxv} becomes simpler, and thus the assumption $k\geq m>3$ in Theorem~\ref{TH3.3} can be relaxed to $k\geq2$ and $m>1$.
\end{remark}

\subsection{Uniqueness of numerical solutions near a specified true solution}
\label{sec:3.3}
To establish the uniqueness for numerical solutions of the interpolated coefficient SGSEM corresponding to a specified true solution of \eqref{HE1}, the following lemma is needed.

\begin{lemma}\label{LM3.3}
Assume that ${\bf(A_1)}$, ${\bf(A_2)}$ and ${\bf(A_3)}$ hold and $u\in H^{m}(\Omega)\cap H_0^1(\Omega)$ ($m>1$) is a solution of \eqref{HE1}. Then, there exist constants $c_8>0$ and $M>0$ such that if $N>M$, it holds
\begin{equation}
\|v_N\|_1\leq c_8\sup_{0\neq\varphi\in {X_{N}}}\frac{\mathcal{B}\big(u;v_N,\varphi\big)}{\|\varphi\|_{1}}, \quad \forall v_N\in X_N.
\end{equation}
\end{lemma}
\begin{proof}
First, by \eqref{iso}, we have
\begin{equation}\label{iso1}
\|v_N\|_1\leq C\|\mathcal{L}'(u)v_N\|_{-1}
  = C\sup_{0\neq \varphi\in H_{0}^{1}(\Omega)}\frac{\mathcal{B}(u;v_N,\varphi)}{\|\varphi\|_{1}}, \quad \forall v_N\in X_N.
\end{equation}
Note that $u\in H^m(\Omega)\hookrightarrow C(\bar{\Omega})$. Hypothesis ${\bf(A_1)}$, orthogonality \eqref{eq:h1proj} and Lemma~\ref{LM:PI} imply that for any $\varphi\in H_0^1(\Omega)$,
\begin{align*}
\mathcal{B} \big(u;v_N,\mathrm{\Pi}_N^{1,0}\varphi \big)
&= \mathcal{B} \big(u;v_N,\varphi \big)- \mathcal{B} \big(u;v_N,\varphi-\mathrm{\Pi}_N^{1,0}\varphi \big)\\
&= \mathcal{B} \big(u;v_N,\varphi \big)+\big(f'(u)v_N,\varphi-\mathrm{\Pi}_N^{1,0}\varphi \big)\\
&\geq \mathcal{B} \big(u;v_N,\varphi \big)-CN^{-1}\|v_N\|_1\|\varphi\|_1.
\end{align*}
Hence, by (\ref{iso1}), it holds
\begin{equation}\label{iso2}
 \sup_{0\neq\varphi\in H_{0}^{1}(\Omega)}\frac{\mathcal{B}\big(u;v_N,\mathrm{\Pi}_N^{1,0}\varphi \big)}{\|\varphi\|_1}
 \geq C\|v_N\|_1,
\end{equation}
for $N$ large enough. With the fact that $\big\|\mathrm{\Pi}_N^{1,0}\varphi\big\|_1\leq C\|\varphi\|_1$ and $\mathrm{\Pi}_N^{1,0}$ is a surjective, one obtains
\begin{equation}\label{iso3}
 \sup_{0\neq\varphi\in {X_{N}}}\frac{\mathcal{B}\big(u;v_N,\varphi\big)}{\|\varphi\|_1}
 =\sup_{0\neq\varphi\in H_{0}^{1}(\Omega)}\frac{\mathcal{B}\big(u;v_N,\mathrm{\Pi}_N^{1,0}\varphi\big)}{\|\mathrm{\Pi}_N^{1,0}\varphi\|_1}
\geq C \sup_{0\neq\varphi\in H_{0}^{1}(\Omega)}\frac{\mathcal{B}\big(u;v_N,\mathrm{\Pi}_N^{1,0}\varphi\big)}{\| \varphi \|_1}.
\end{equation}
Combining \eqref{iso2} and \eqref{iso3} completes the proof.
\end{proof}

\begin{lemma}\label{LM3.4}
Assume that ${\bf(A_1)}$, ${\bf(A_2)}$ and ${\bf(A_3)}$ hold and $u\in H^{m}(\Omega)\cap H_0^1(\Omega)$ ($m>1$) is a solution of \eqref{HE1}. Then, there exist constants $\delta,M, c_9>0$ such that for any $w\in C(\bar{\Omega})$ with $\|u-w\|_{L^{\infty}(\Omega)}\leq \delta$ and $N>M$, it yields
\begin{equation}
\|v_N\|_1\leq c_9 \sup_{0\neq\varphi\in {X_{N}}}\frac{\mathcal{B}\big(w;v_N,\varphi \big)}{\|\varphi\|_1}, \quad
\forall v_N\in {X_{N}}.
\end{equation}
\end{lemma}
\begin{proof}
Taking $\varphi\in X_{N}$, we have
\begin{align*}
\mathcal{B}\big(w;v_N,\varphi \big)
&=\mathcal{B} \big(u;v_N,\varphi \big)+\big(\big(f'(u)-f'(w)\big)v_N,\varphi \big)\\
&\geq \mathcal{B} \big(u;v_N,\varphi \big)-\big\|f'(u)-f'(w)\big\|_{L^{\infty}(\Omega)} \|v_N\|_1 \|\varphi\|_1, \quad \forall v_N\in X_N.
\end{align*}
Since $f'$ is continuous from ${\bf(A_1)}$, we have for any $\varepsilon>0$, there exists a $\delta>0$, such that for any $w\in C(\bar{\Omega})$ with $\|u-w\|_{L^{\infty}(\Omega)}\leq\delta$, it holds $\big\|f'(u)-f'(w)\big\|_{L^{\infty}(\Omega)}\leq\varepsilon$. Thus
$$
\mathcal{B}\big(w;v_N,\varphi\big)
 \geq \mathcal{B}\big(u;v_N,\varphi\big)-\varepsilon\|v_N\|_1\|\varphi\|_1, \quad \forall v_N,\varphi\in X_N.
$$
The statement follows by virtue of Lemma~\ref{LM3.3} with $\varepsilon=\frac{1}{2c_8}$ and $c_9=2c_8$.
\end{proof}

For any fixed $w\in C(\bar{\Omega})$, define
\begin{equation}\label{h542}
\mathcal{B}_N\big(w;v,\varphi\big)=\mathcal{A}(v,\varphi)-\big(\mathrm{I}_N\big(f'(w)v\big),\varphi\big),
\quad\forall v\in {X_N},\varphi\in C(\bar{\Omega}).
\end{equation}

\begin{lemma}\label{LM3.5}
Assume that ${\bf(A_1)}$, ${\bf(A_2)}$ and ${\bf(A_3)}$ hold and $u\in H^{m}(\Omega)\cap H_0^1(\Omega)$ ($m>3$) is a solution of \eqref{HE1}. Then, there exist constants $\delta_1,M>0$ such that, if $w\in C(\bar{\Omega})$ with $\|u-w\|_{L^{\infty}(\Omega)}\leq\delta_1$ and $N>M$, it yields
\begin{equation}\label{h543}
\|v_N\|_1
  \leq C\sup_{0\neq\varphi\in {X_{N}}}
       \frac{\mathcal{B}_N \big(w;v_N,\varphi\big)}{\|\varphi\|_1},
        \quad \forall v_N\in {X_{N}}.
\end{equation}
\end{lemma}
\begin{proof}
Taking $\varphi\in X_{N}$, we have
\begin{align}
\mathcal{B}_N\big(w;v_N,\varphi\big)
 &= \mathcal{B}\big(w;v_N,\varphi\big)+\big(f'(w)v_N-\mathrm{I}_N\big(f'(w)v_N\big),\varphi \big) \nonumber \\
 &\geq \mathcal{B}\big(w;v_N,\varphi\big) - \big\|f'(w)v_N-f'(u)v_N\big\|_{0} \|\varphi\|_1 \nonumber \\
 &\quad\; -\big\|f'(u)v_N-\mathrm{I}_N\big(f'(u)v_N\big)\big\|_0\|\varphi\|_1 \nonumber \\
 &\quad\; - \big\|\mathrm{I}_N\big(f'(u)v_N\big)-\mathrm{I}_N\big(f'(w)v_N\big)\big\|_0\|\varphi\|_1. \label{eq:h337}
\end{align}
Since $f'$ is continuous from ${\bf(A_1)}$, we have for any $\varepsilon>0$, there exists a $\delta_0>0$, such that for any $w\in C(\bar{\Omega})$ with $\|u-w\|_{L^{\infty}(\Omega)}\leq\delta_0$,
\begin{equation}\label{eq:lm3.5:2}
 \big\|f'(w)v_N-f'(u)v_N\big\|_{0}\leq \big\|f'(w)-f'(u)\big\|_{L^{\infty}(\Omega)} \| v_N\|_1 \leq \varepsilon \| v_N \|_1.
\end{equation}
Moreover, for any $w\in C(\bar{\Omega})$ with $\|u-w\|_{L^{\infty}(\Omega)}\leq\delta_0$,
\begin{align}
\big\|\mathrm{I}_N\big(f'(u)v_N\big)-\mathrm{I}_N\big(f'(w)v_N\big) \big\|_0&\leq \big\|\mathrm{I}_N\big(f'(u)v_N\big)-\mathrm{I}_N\big(f'(w)v_N\big)\big\|_{N} \nonumber \\
&= \big\|f'(u)v_N-f'(w)v_N \big\|_{N}\nonumber \\
&\leq \big\|f'(w)-f'(u) \big\|_{L^{\infty}(\Omega)}\|v_N\|_N\nonumber\\
&\leq  c_L \varepsilon \|v_N\|_1. \label{eq:h339}
\end{align}
Afterwards, by virtue of $(\bf A_1)$, Lemma~\ref{LM-interopation-sp} and Lemma~\ref{LM-inverse}, we have for any fixed $\tau \in(0,1)$,
\begin{align*}
\big\|f'(u)v_N-\mathrm{I}_N\big(f'(u)v_N\big)\big\|_0
&\leq c_2 N^{-(1+\tau)}\big\|f'(u)v_N\big\|_{1+\tau}\\
&\leq C N^{-(1+\tau)}\big\|f'(u)v_N\big\|_{[H^2(\Omega),H^1(\Omega)]_{1-\tau}}\\
&\leq CN^{-(1+\tau)}\big\|f'(u)v_N\big\|_2^{\tau}\big\|f'(u)v_N\big\|_1^{1-\tau}\\
&\leq CN^{-(1+\tau)}\big\|v_N\big\|_2^{\tau}\big\|v_N\big\|_1^{1-\tau}\\
&\leq CN^{\tau-1}\|v_N\|_1,
\end{align*}
where the facts that $\big[H^2(\Omega),H^1(\Omega)\big]_{1-\tau}=H^{1+\tau}(\Omega)$ and $u\in H^m(\Omega)\hookrightarrow C^2(\bar{\Omega})$ $(m>3)$ are applied. Thus, for sufficiently large $N$, it yields
\begin{equation}\label{eq:h340}
\big\|f'(u)v_N-\mathrm{I}_N\big(f'(u)v_N\big)\big\|_0 \leq \varepsilon\|v_N\|_1.
\end{equation}
Combining (\ref{eq:h337})-(\ref{eq:h340}) leads to
$$
\mathcal{B}_N\big(w;v_N,\varphi\big)
 \geq \mathcal{B}\big(w;v_N,\varphi\big) - (2+c_L)\varepsilon\|v_N\|_1\|\varphi\|_1.
$$
Taking $\varepsilon=\frac{1}{2(2+c_L)c_9}$ and $\delta_1=\min\{\delta_0,\delta\}$ with $c_9$ and $\delta$ two constants in Lemma~\ref{LM3.4}, one obtains the result immediately.
\end{proof}

\begin{theorem}\label{TH3.4}
Assume that ${\bf(A_1)}$, ${\bf(A_2)}$ and ${\bf(A_3)}$ hold and $u\in H^{m}(\Omega)\cap H_0^1(\Omega)$ ($k\geq m>3$) is a solution of \eqref{HE1}. Then, for sufficiently large $N$, there exists a unique solution $u_N$ of \eqref{eq:HE3-weak} in the neighborhood
$$
\mathcal{N}_{\delta_1}(u)\cap{X_N}
=\big\{v\in X_N:\|v-u\|_{L^{\infty}(\Omega)}\leq\delta_1\big\},
$$
satisfying the error estimates \eqref{h531}, where $\delta_1>0$ is the constant in Lemma~\ref{LM3.5}.
\end{theorem}
\begin{proof}
By Lemma~\ref{LM:PI-highorder}, Lemma~\ref{LM-inverse}, and Theorem~\ref{TH3.3}, for every sufficiently large $N$, there exists a $u_N\in X_N$ such that, for $1<s<3/2$,
\begin{align*}
\|u-u_N\|_{L^{\infty}(\Omega)}
 &\leq C\|u-u_N\|_s \\
 &\leq C\big\|u-\mathrm{\Pi}_N^{s,0}u\big\|_s+\big\|\mathrm{\Pi}_N^{s,0}u-u_N\big\|_s \\
 &\leq CN^{s-m}+CN^{2(s-1)}\big\|\mathrm{\Pi}_N^{s,0}u-u_N\big\|_1 \\
 &\leq CN^{s-m}+CN^{2(s-1)}\left(\big\|u-\mathrm{\Pi}_N^{s,0}u\big\|_1+\|u-u_N\|_1\right) \\
 &\leq CN^{s-m}+CN^{2(s-1)}N^{1-m} \\
 &=\mathrm{o}(1).
\end{align*}
This implies that $u_N\in\mathcal{N}_{\delta_1}(u)\cap{X_N}$ for $N$ large enough.

It remains to verify the uniqueness of solutions to \eqref{eq:HE3-weak} in the neighborhood $\mathcal{N}_{\delta_1}(u)\cap{X_N}$. Assume that $u_N,\tilde{u}_N\in \mathcal{N}_{\delta_1}(u)\cap{X_N}$ are two solutions of \eqref{eq:HE3-weak}, we have
$$
u_N+t(\tilde{u}_N-u_N)\in \mathcal{N}_{\delta_1}(u)\cap{X_{N}}, \quad \forall t\in[0,1].
$$
Note that for any $w,\bar{w}\in\mathcal{N}_{\delta_1}(u)\cap{X_N}$,
\begin{align*}
 \big|\mathcal{B}_N\big(w;v_N,\varphi\big)-\mathcal{B}_N\big(\bar{w};v_N,\varphi\big)\big|
 &= \big|\big(\mathrm{I}_N(f'(w)v_N)-\mathrm{I}_N(f'(\bar{w})v_N),\varphi\big)\big| \\
 &\leq \big\|f'(w)v_N-f'(\bar{w})v_N\big\|_N \|\varphi\|_0 \\
 &\leq c_L\big\|f'(w)-f'(\bar{w})\big\|_{L^{\infty}(\Omega)} \|v_N\|_0\|\varphi\|_0,
 \quad \forall v_N,\varphi\in X_N.
\end{align*}
The regularity of $f$ implies that $\mathcal{B}_N\big(w;v_N,\varphi\big)$ is continuous with respect to $w$ in $\mathcal{N}_{\delta_1}(u)\cap{X_N}$. Thus, $\mathcal{B}_N\big(u_N+t(\tilde{u}_N-u_N);\tilde{u}_N-u_N,\varphi\big)$ is continuous with respect to $t\in[0,1]$. This means that for any $\varphi\in X_N$, there exists a $\bar{t}\in[0,1]$ such that
\begin{align*}
\int_0^1\mathcal{B}_N \big(u_N+t(\tilde{u}_N-u_N);\tilde{u}_N-u_N,\varphi \big)dt=\mathcal{B}_N\big(u_N+\bar{t}(\tilde{u}_N-u_N);\tilde{u}_N-u_N,\varphi\big).
\end{align*}
On the other hand,  a direct calculation yields
\begin{align*}
&\;\quad \int_0^1\mathcal{B}_N \big(u_N+t(\tilde{u}_N-u_N);\tilde{u}_N-u_N,\varphi \big)dt \\
&=\mathcal{A}(\tilde{u}_N-u_N,\varphi)-\left(\mathrm{I}_N\int_0^1f'\big(u_N+t(\tilde{u}_N-u_N)\big)(\tilde{u}_N-u_N)dt,\varphi\right) \\
&=\mathcal{A}(\tilde{u}_N-u_N,\varphi)-\big(\mathrm{I}_Nf(\tilde{u}_N)-\mathrm{I}_Nf(u_N),\varphi\big)=0,
\quad \forall \varphi\in X_N.
\end{align*}
Thus, for any $\varphi\in X_N$, there exists a $\bar{t}\in[0,1]$ such that $\mathcal{B}_N\big(u_N+\bar{t}(\tilde{u}_N-u_N);\tilde{u}_N-u_N,\varphi\big)=0$. Further, by Lemma~\ref{LM3.5}, it holds
\begin{equation}\label{eq:uniquess2}
\|u_N-\tilde{u}_N\|_1 \leq C \sup_{0\neq\varphi\in{X_N}} \frac{\mathcal{B}_N\big(u_N+\bar{t}(\tilde{u}_N-u_N);\tilde{u}_N-u_N,\varphi\big)}{\|\varphi\|_1} =0.
\end{equation}
Consequently, $u_N=\tilde{u}_N$. Theorem~\ref{TH3.4} is proved.
\end{proof}

\section{Numerical experiments}
\label{sec:4}
In this section, we apply the SGSEM to compute multiple solutions of the model problem \eqref{HE1}. The computational domain is taken as $\Omega=(0,\pi)^2$. Then, all the eigenvalues and the corresponding orthonormal eigenfunctions of $-\Delta$ in $\Omega$ with the homogeneous Dirichlet boundary condition are
$$
  \lambda_{mn}=m^2+n^2,\quad \phi_{mn}(\x)=(2/\pi)\sin(mx)\sin(ny),\quad m,n=1,2,\cdots.
$$
We usually take $N=32$ in our numerical computation unless specified.

\subsection{The case of cubic nonlinearity}

First, we consider the case of cubic nonlinearity to facilitate comparison with existing theoretical \cite{ZYZ13} and numerical results \cite{CX,XC,ZYZ13}. Taking $f(u)=u^3$, the model problem \eqref{HE1} becomes
\begin{align}\label{eq:cubic}
 \left\{\begin{array}{rll}
  \Delta u+u^3=0 &\mbox{in}& \Omega,\\
  u=0 & \mbox{on} & \partial\Omega.
 \end{array}\right.
\end{align}
Equation \eqref{eq:cubic} is often referred to as the Lane-Emden equation. It is known in the literature that the structure and distribution of the multiple solutions of \eqref{eq:cubic} are very complex. Actually, \eqref{eq:cubic} has an infinite number of solutions \cite{cys}. In order to compute multiple solutions of \eqref{eq:cubic}, we construct initial guesses according to the multiplicity of eigenvalues of $-\Delta$. More precisely, in Step~2 of the algorithm of SGSEM, we search all solutions to the following subproblem: find $u^0\in S_\lambda$ such that
\begin{equation}\label{e58}
(\nabla u^0,\nabla v)=\big((u^0)^3,v\big),\quad \forall v\in S_\lambda,
\end{equation}
where $S_\lambda$ is the subspace of $H_0^1(\Omega)$ spanned by all eigenfunctions corresponding to a given eigenvalue $\lambda$. According to \cite{CX,XC,ZYZ13}, when $\lambda$ is a $q$-fold eigenvalue with corresponding eigenfunctions $\phi_1,\phi_2,\cdots,\phi_q$, the nonlinear problem \eqref{e58} has exactly $3^q-1$ nontrivial solutions of the form $u^0=a_1\phi_1+a_2\phi_2+\cdots+a_q\phi_q$, where
\begin{equation}
  a_i=\pm 4\sqrt{\lambda/(3(4r-1))},\quad\forall i\in\mathcal{I};\qquad
  a_i=0,\quad \forall i\in\{1,2,\cdots,q\}\backslash\mathcal{I},
\end{equation}
with $\mathcal{I}$ denoting any nonempty subset of $\{1,2,\cdots,q\}$ and $r=|\mathcal{I}|$ the cardinality of $\mathcal{I}$. The numerical results show that by searching for initial guesses in the above manner, multiple solutions to the model problem \eqref{eq:cubic} can be efficiently computed according to specific rules, such as the order of eigenvalues or their multiplicities from small to large. Moreover, only a small nonlinear algebraic system needs to be solved when searching for initial guesses. In fact, these are the two advantages of the SEM for computing multiple solutions. However, when $f(u)$ in \eqref{HE1} is more complicated, it is generally necessary to use Step~3 in the SGSEM algorithm to get a better initial guess.

{\bf (i) The single-fold eigenvalue case.}
By taking $\lambda_{mm}=2m^2$ and $u^0=a\phi_{mm}$ in \eqref{e58}, one obtains two nontrivial roots $a=\pm4\sqrt{2}m/3$. Then we solve for the model problem \eqref{HE1} by the SGSEM with the initial guess $u_{mm}^0=a\phi_{mm}$. Denote the corresponding solution by $u_{mm}$. Fig.~\ref{fig:u112244} presents the pcolor profiles of three solutions $u_{11}$, $u_{22}$ and $u_{44}$ and their initial guesses. One observes that each solution is very similar to the corresponding initial guess in shape, whereas possesses higher peaks and slenderer ``waists''.

\begin{figure}[!ht]
  \centering
  \quad\includegraphics[width=1.6in,height=1.0in]{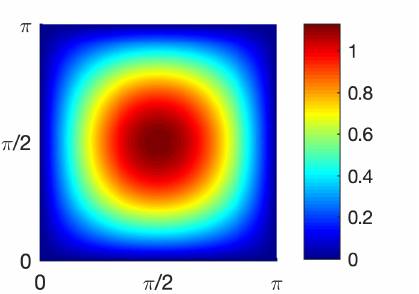}
  \quad\includegraphics[width=1.6in,height=1.0in]{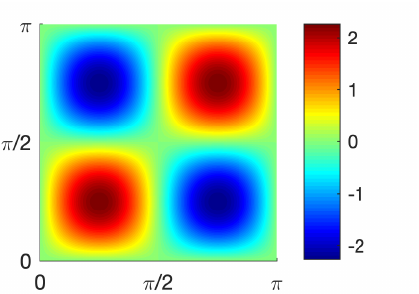}
  \quad\includegraphics[width=1.6in,height=1.0in]{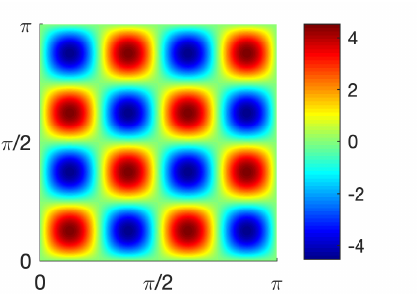} \\
  \quad\includegraphics[width=1.6in,height=1.0in]{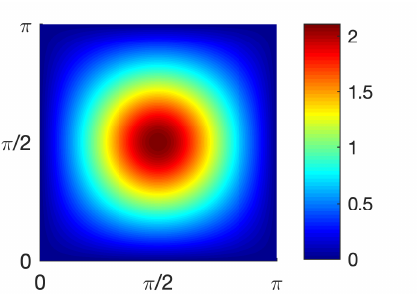}
  \quad\includegraphics[width=1.6in,height=1.0in]{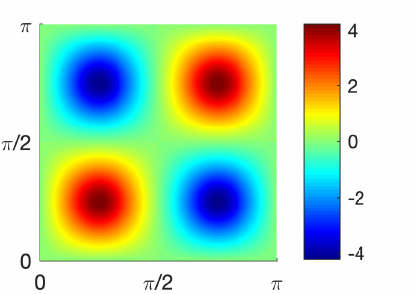}
  \quad\includegraphics[width=1.6in,height=1.0in]{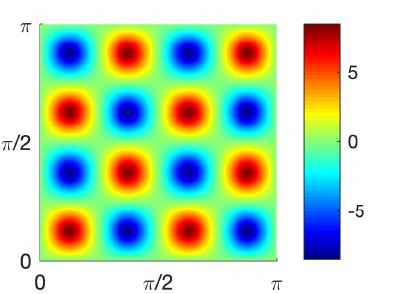}
  \caption{Three initial guesses $u^0_{mm}$ (top row) and the corresponding solutions $u_{mm}$ (bottom row) for $m=1$ (left column), $m=2$ (middle column) and $m=4$ (right column).}
  \label{fig:u112244}
\end{figure}

{\bf (ii) The 2-fold eigenvalue case.}
Consider the 2-fold eigenvalue $\lambda_{mn}=\lambda_{nm}=m^2+n^2$ for two different positive integers $m$ and $n$. Taking $u^0=a\phi_{mn}+b\phi_{nm}$ in \eqref{e58} leads to eight nontrivial roots: $a=\pm\alpha,b=0$; $a=0,b=\pm\alpha$; $a=b=\pm\beta$; $a=-b=\pm\beta$. Here $\alpha=4\sqrt{\lambda_{mn}}/3$ and $\beta=4\sqrt{\lambda_{mn}/21}$. Thus, there exist eight initial guesses, i.e., $\pm\alpha\phi_{mn}$, $\pm\alpha\phi_{nm}$, $\pm\beta(\phi_{mn}+\phi_{nm})$ and $\pm\beta(\phi_{mn}-\phi_{nm})$, which correspond to eight solutions of \eqref{HE1}. Due to the symmetries of the model problem \eqref{HE1}, we only present three solutions $u_{mn}$, $u_{mn+nm}$ and $u_{mn-nm}$ with respect to the initial guesses $\alpha\phi_{mn}$, $\beta(\phi_{mn}+\phi_{nm})$ and $\beta(\phi_{mn}-\phi_{nm})$, respectively.  Fig.~\ref{fig:u121315} presents the solutions with $(m,n)=(1,2)$, $(1,3)$ and $(1,5)$.

\begin{figure}[!t]
  \centering
  \quad\includegraphics[width=1.6in,height=1.0in]{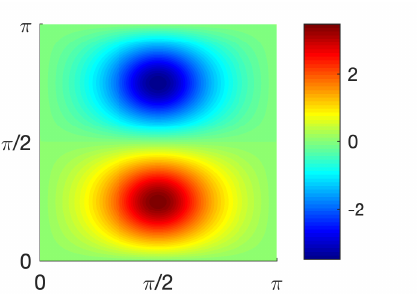}
  \quad\includegraphics[width=1.6in,height=1.0in]{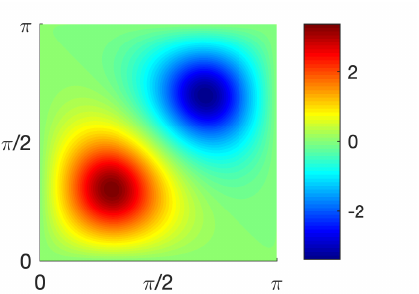}
  \quad\includegraphics[width=1.6in,height=1.0in]{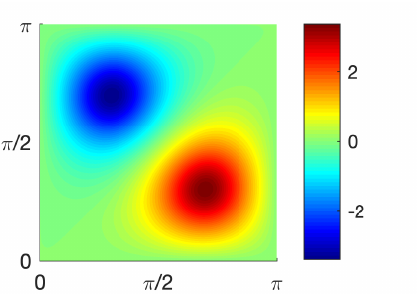} \\
  \quad\includegraphics[width=1.6in,height=1.0in]{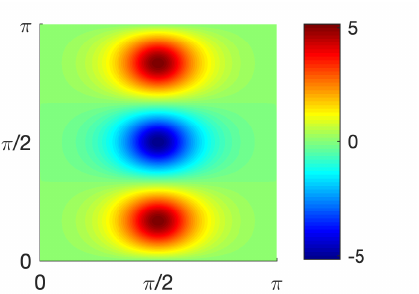}
  \quad\includegraphics[width=1.6in,height=1.0in]{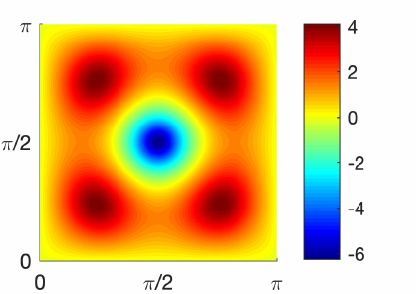}
  \quad\includegraphics[width=1.6in,height=1.0in]{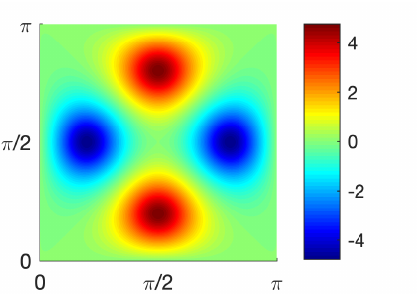} \\
  \quad\includegraphics[width=1.6in,height=1.0in]{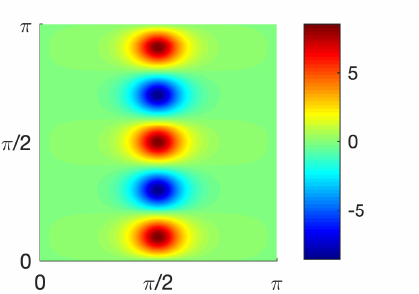}
  \quad\includegraphics[width=1.6in,height=1.0in]{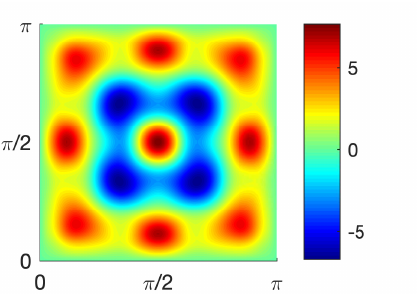}
  \quad\includegraphics[width=1.6in,height=1.0in]{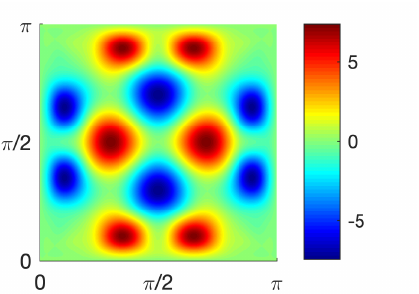}
  \caption{The solutions $u_{mn}$ (left column), $u_{mn+nm}$ (middle column) and $u_{mn-nm}$ (right column) for $(m,n)=(1,2)$ (top row), $(m,n)=(1,3)$ (middle row) and $(m,n)=(1,5)$ (bottom row).}
  \label{fig:u121315}
\end{figure}

{\bf (iii) The 3-fold eigenvalue case.}
Consider the smallest 3-fold eigenvalue $\lambda_{17}=\lambda_{71}=\lambda_{55}=50=\lambda$. Taking $u^0=a\phi_{17}+b\phi_{71}+c\phi_{55}$ in \eqref{e58}, one obtains $3^3-1=26$ different initial guesses. Fig.~\ref{fig:u175571} shows six solutions of \eqref{HE1} with respect to this eigenvalue.

\begin{figure}[!t]
  \centering
  (a)\includegraphics[width=1.6in,height=1.0in]{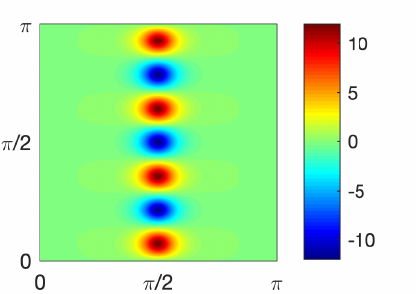}
  (b)\includegraphics[width=1.6in,height=1.0in]{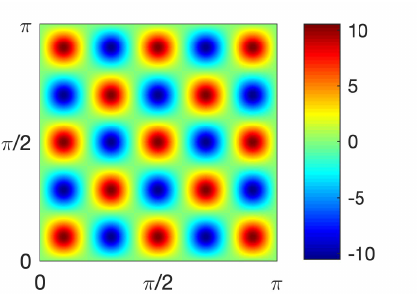}
  (c)\includegraphics[width=1.6in,height=1.0in]{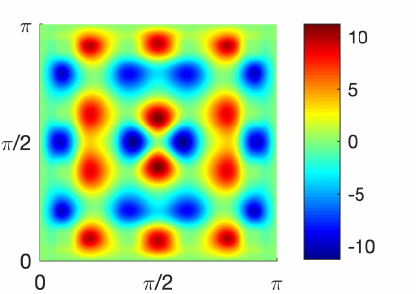} \\
  (d)\includegraphics[width=1.6in,height=1.0in]{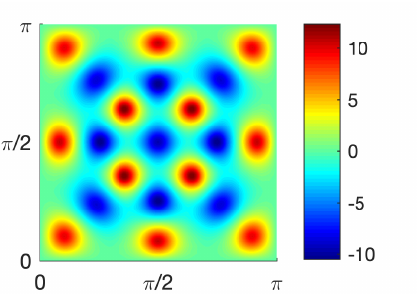}
  (e)\includegraphics[width=1.6in,height=1.0in]{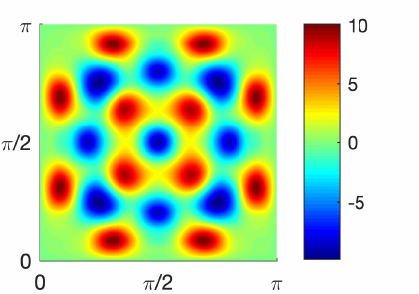}
  (f)\includegraphics[width=1.6in,height=1.0in]{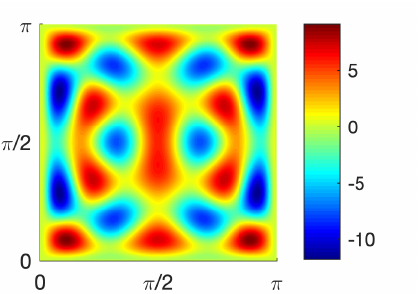}
  \caption{Six solutions corresponding to the least 3-fold eigenvalue $\lambda=50$: (a) $u_{17}$, (b) $u_{55}$, (c) $u_{17-71}$, (d) $u_{17+55+71}$, (e) $u_{17-55+71}$ and (f) $u_{17+55-71}$.}
  \label{fig:u175571}
\end{figure}

{\bf (iv) The 4-fold eigenvalue case.}
Consider the smallest 4-fold eigenvalue $\lambda_{18}=\lambda_{81}=\lambda_{47}=\lambda_{74}=65=\lambda$. Taking $u^0=a\phi_{18}+b\phi_{81}+c\phi_{47}+d\phi_{74}$ in \eqref{e58}, one gets $3^4-1=80$ different initial guesses. Three solutions corresponding to this 4-fold eigenvalue are computed with $N=42$ and shown in Fig.~\ref{fig:u18477481}.

\begin{figure}[!t]
  \centering
  \quad\includegraphics[width=1.6in,height=1.0in]{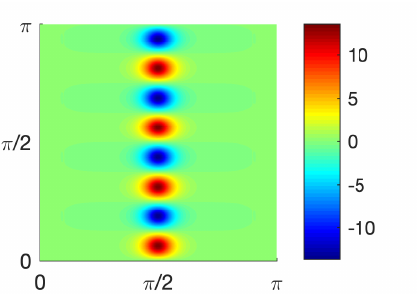}
  \quad\includegraphics[width=1.6in,height=1.0in]{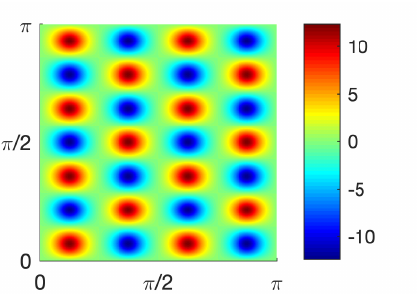}
  \quad\includegraphics[width=1.6in,height=1.0in]{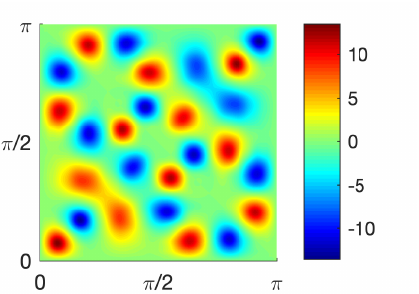}
  \caption{Three solutions corresponding to the least 4-fold eigenvalue $\lambda=65$: $u_{18}$ (left), $u_{47}$ (middle) and $u_{47+74}$ (right).}
  \label{fig:u18477481}
\end{figure}

Next, we illustrate the spectral convergence rate of the SGSEM. Without loss of generality, we consider the errors for the numerical solutions $u_{11}$ and $u_{12+21}$. Notice that the ``exact" solutions of $u_{11}$ and $u_{12+21}$ are computed with a very large spectral-Galerkin approximation space, say $N=100$. Fig.~\ref{fig:u_errs} demonstrates the $L^2$- and $H^1$-errors for $u_{11}$ and $u_{12+21}$ in the semilogarithmic scale, from which the exponential convergence rates are observed. This indicates that our approach is efficient and exponentially convergent for computing multiple solutions.

\begin{figure}[!t]
  \centering
  \includegraphics[width=.4\textwidth]{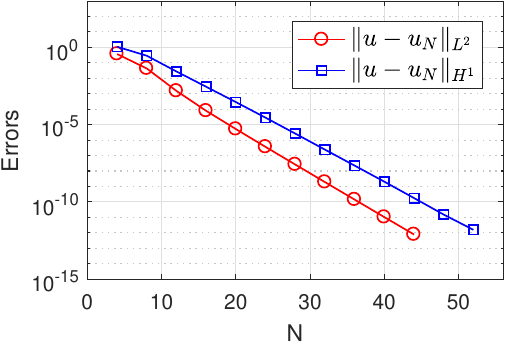}\qquad
  \includegraphics[width=.4\textwidth]{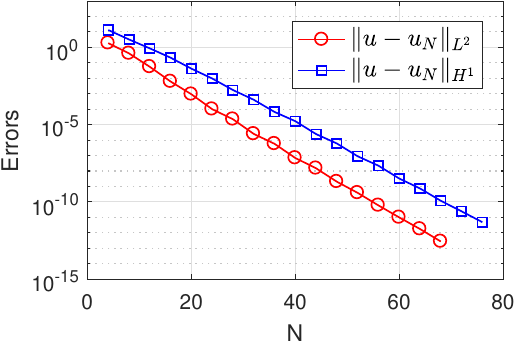}
  \caption{Exponential convergence rate in $L^2(\Omega)$ and $H^1(\Omega)$ of SGSEM for solutions $u=u_{11}$ (left) and $u=u_{12+21}$ (right).}
  \label{fig:u_errs}
\end{figure}

Further, we compare the efficiency of the SEM using the interpolated coefficient Legendre-Galerkin spectral method (ICLG) and the interpolated coefficient bilinear rectangular finite element method (ICFEM) \cite{XC}. Both algorithms use the same initial guess, Newton-type numerical extension method, and residual termination criterion (i.e., the Newton iteration is terminated when the discrete $l^2$ error of the residual is less than $10^{-10}$) to compute the same solution. Tables \ref{tab:err-u11} and \ref{tab:err-u12+21} list the errors of $u_{11}$ and $u_{12+21}$ computed by the two algorithms under different grids, respectively. The numerical results show that the $L^2$- and $H^1$-errors corresponding to the ICFEM have second- and first-order accuracy, respectively, while the ICLG has spectral convergence accuracy. In addition, for the computation of $u_{11}$, the SEM based on the ICFEM discretization uses $48\times48$ degrees of freedom making the $L^2$- and $H^1$-errors reach $1.85\times 10^{-3}$ and $6.28\times 10^{-2}$, respectively, with a total computational time of $4.36$ seconds, while the SEM based on the ICLG discretization uses only $16\times16$ degrees of freedom making the $L^2$- and $H^1$-errors reach $7.8\times 10^{-5}$ and $2.81\times 10^{-3}$, respectively, with a total computational time of $0.08$ seconds. Similar results are also observed in the computation of $u_{12+21}$. These observations show that, in terms of the degrees of freedom and CPU times required for the numerical solution to reach the same level of accuracy, the efficiency of the SEM based on the ICLG discretization is significantly higher than that of the SEM based on the ICFEM discretization.

\begin{table}[!t]
\centering\small
\caption{The errors of $u=u_{11}$ computed by the two algorithms under different grids.}
\label{tab:err-u11}
\begin{tabular}{cccccccc}
\hline
Algorithm & $N$ & 8 & 16 & 24 & 32 & 40 & 48 \\
\hline
\multirow{2}{*}{ICFEM}
& $\|u-u_N\|_{L^2}$ & 5.7343E-02 & 1.6048E-02 & 7.2947E-03 & 4.1363E-03 & 2.6566E-03 & 1.8488E-03 \\
& $\|u-u_N\|_{H^1}$ & 3.9433E-01 & 1.9156E-01 & 1.2663E-01 & 9.4204E-02 & 7.5320E-02 & 6.2776E-02 \\
\hline
\multirow{2}{*}{ICLG}
& $\|u-u_N\|_{L^2}$ & 4.2857E-02 & 7.8021E-05 & 3.6199E-07 & 1.9073E-09 & 1.0339E-11 & 1.2248E-13 \\
& $\|u-u_N\|_{H^1}$ & 2.8632E-01 & 2.8138E-03 & 2.8109E-05 & 2.4095E-07 & 2.0338E-09 & 1.5095E-11 \\
\hline
\end{tabular}
\end{table}

\begin{table}[!t]
\centering\small
\caption{The errors of $u_{12+21}$ computed by the two algorithms under different grids.}
\label{tab:err-u12+21}
\begin{tabular}{cccccccc}
\hline
Algorithm & $N$ & 20 & 30 & 40 & 50 & 60 & 70 \\
\hline
\multirow{2}{*}{ICFEM}
& $\|u-u_N\|_{L^2}$ & 3.6358E-02 & 1.6732E-02 & 9.5300E-03 & 6.1350E-03 & 4.2741E-03 & 3.1449E-03 \\
& $\|u-u_N\|_{H^1}$ & 6.0492E-01 & 3.9613E-01 & 2.9930E-01 & 2.3636E-01 & 1.9769E-01 & 1.6906E-01 \\
\hline
\multirow{2}{*}{ICLG}
& $\|u-u_N\|_{L^2}$ & 9.2365E-04 & 8.1137E-06 & 6.9518E-08 & 7.1984E-10 & 9.9249E-12 & 4.9298E-13 \\
& $\|u-u_N\|_{H^1}$ & 4.2015E-02 & 8.9500E-04 & 1.6794E-05 & 2.4773E-07 & 3.2591E-09 & 4.4562E-11 \\
\hline
\end{tabular}
\end{table}

\subsection{The case of triangular nonlinearity}

The SGSEM proposed in this paper can also be applied to the computation of multiple solutions of semilinear elliptic problems with other types of nonlinear terms. We now consider a non-polynomial semilinear case. Taking $f(u)=\kappa\sin u$ ($\kappa>0$), the corresponding semilinear problem \eqref{HE1} becomes
\begin{align}\label{eq:sinG}
 \left\{\begin{array}{rll}
  \Delta u+\kappa\sin u=0 &\mbox{in}& \Omega,\\
  u=0 & \mbox{on} & \partial\Omega.
 \end{array}\right.
\end{align}
Equation \eqref{eq:sinG} is usually called the elliptic sine-Gordon equation \cite{CDHNZ2004CM}, and can also be regarded as a generalization of the one-dimensional Euler problem for the bending of compressed bars \cite{CXbook}. Notice that here $\kappa>0$ is a constant and the nonlinear function $f(u)=\kappa\sin u$ is a bounded odd function.

It is worth noting that the existence and distribution of multiple solutions of the boundary value problem \eqref{eq:sinG} are related to the value of the parameter $\kappa$. Arranging the eigenvalues of $-\Delta$ in ascending order:
\[ 0<\lambda_1<\lambda_2\leq\lambda_3\leq\cdots\leq\lambda_j\leq\cdots, \]
the first ten eigenvalues are listed in Table~\ref{tab:sinG_mnl}. According to \cite{CDHNZ2004CM}, when $0<\kappa<\lambda_1$, the boundary value problem \eqref{eq:sinG} has only zero solution $u\equiv0$; when $\kappa>\lambda_1$, \eqref{eq:sinG} has at least two nontrivial solutions, including a positive solution $u$ and its corresponding negative solution $-u$; when $\kappa>\lambda_2=\lambda_3$, \eqref{eq:sinG} has at least six nontrivial solutions and sign-changing solutions appear; and generally, when $\kappa>\lambda_k$ ($k\geq1$), \eqref{eq:sinG} has at least $k$ pairs of nontrivial solutions (i.e., $2k$ nontrivial solutions).

\begin{table}[!ht]
\centering\small
\caption{The first ten eigenvalues of $-\Delta$, $\lambda_j=m_j^2+n_j^2$ ($j=1,2,\ldots,10$), and the corresponding $(m_j,n_j)$.}
\label{tab:sinG_mnl}
\begin{tabular}{ccccccccccc}
\hline
$j$ & 1 & 2 & 3 & 4 & 5 & 6 & 7 & 8 & 9 & 10 \\
\hline
$m_j$ & 1 & 2 & 1 & 2 & 3 & 1 & 3 & 2 & 4 & 1 \\
$n_j$ & 1 & 1 & 2 & 2 & 1 & 3 & 2 & 3 & 1 & 4 \\
$\lambda_j=m_j^2+n_j^2$ & 2 & 5 & 5 & 8 & 10 & 10 & 13 & 13 & 17 & 17 \\
\hline
\end{tabular}
\end{table}

In the following, we use the SGSEM to compute multiple solutions to the problem \eqref{eq:sinG} with different values of the parameter $\kappa$. The following five types of values of $\kappa$ are considered: 
\begin{enumerate}[(i)]
\setlength\itemsep{-.6ex}
\item $\kappa=1$ ($0<\kappa<\lambda_1$);
\item $\kappa=4$ ($\lambda_1<\kappa<\lambda_2=\lambda_3$);
\item $\kappa=6$ ($\lambda_3<\kappa<\lambda_4$);
\item $\kappa=9$ ($\lambda_4<\kappa<\lambda_5=\lambda_6$);
\item $\kappa=11$ ($\lambda_6<\kappa<\lambda_7$).
\end{enumerate}
In Steps 2 and 3 of the SGSEM algorithm, we directly use the linear combinations of eigenfunctions corresponding to the first ten eigenvalues of $-\Delta$ to search for the initial guesses of multiple solutions. Setting
\[ u^0=\sum_{j=1}^{10}a_j\phi_j, \]
where
\[ \phi_j=\frac{2}{\pi}\sin(m_jx)\sin(n_jy), \]
with informations for $(m_j,n_j)$ listed in Table~\ref{tab:sinG_mnl}, find the coefficients $(a_1,a_2,\ldots,a_{10})$ such that
\begin{align*}
  \left(\nabla u^0,\nabla v\right)-\kappa\left(\sin u^0,v\right)=0,\quad \forall v\in\spn\{\phi_1,\phi_2,\ldots,\phi_{10}\},
\end{align*}
i.e.,
\begin{align}\label{eq:sinG-aeq}
  \lambda_ia_i-\kappa\bigg(\sin\bigg(\sum_{j=1}^{10}a_j\phi_j\bigg),\phi_i\bigg) = 0,\quad i=1,2,\ldots,10.
\end{align}
For different values of the parameter $\kappa$, we apply a random Newton method (i.e., the Newton iterative method with randomly selected initial guesses) to search for nonzero solutions of the nonlinear system \eqref{eq:sinG-aeq} in the range of $|a_j|\leq 10$. The results are listed in Table~\ref{tab:sinG_a}, where each set of coefficients with opposite signs is also a solution to \eqref{eq:sinG-aeq}. Our computations show that the initial guess corresponding to each set of coefficients can effectively obtain the corresponding numerical solution of \eqref{eq:sinG} through Steps 4 and 5 of the SGSEM algorithm. Numerical solutions corresponding to $\kappa=6$ and $\kappa=11$ are plotted in Figs.~\ref{fig:sinG_k6_5sols} and \ref{fig:sinG_k11_10sols}, respectively.

From Table~\ref{tab:sinG_a} and Figs.~\ref{fig:sinG_k6_5sols}-\ref{fig:sinG_k11_10sols}, we have the following observations:

(1) When $0<\kappa<\lambda_1$, there is no nonzero solution.

(2) The number of solutions increases as the parameter $\kappa$ increases (see Table~\ref{tab:sinG_a}). In view of this, we propose the following conjecture:

{\bf Conjecture 1.} When $\kappa>\lambda_k$ ($k\geq1$), if the number of $q$-fold ($q=1,2,\ldots$) eigenvalues in $\{\lambda_1,\ldots,\lambda_k\}$ is $n_q^{(k)}\geq0$, the number of nontrivial solutions to the problem \eqref{eq:sinG} is at least $\sum_{q=1}^{\infty} n_q^{(k)}(3^q-1)$.

Obviously, when $\kappa>\lambda_k$ ($k\geq1$) and $\{\lambda_1,\ldots,\lambda_k\}$ contains at least a $q^*$-fold ($q^*\geq2$) eigenvalue, the number of nontrivial solutions should be greater than $2k$ due to $\sum_{q=1}^{\infty} n_q^{(k)}=k$ and $3^{q^*}-1>2$.

(3) Each numerical solution possesses the characteristics of ``lower peaks" and ``thicker waists" compared with the principal term of the corresponding initial guess. This is exactly the opposite of the characteristics of solutions to the cubic nonlinear problem.

(4) When $\kappa>\lambda_1$, the model problem has a unique positive solution (see Figs.~\ref{fig:sinG_k6_5sols}(a) and \ref{fig:sinG_k11_10sols}(a)), correspondingly, a unique negative solution, and all other nontrivial solutions are sign-changing. It can be observed from Figs.~\ref{fig:sinG_k6_5sols}(a) and \ref{fig:sinG_k11_10sols}(a) that when the parameter $\kappa$ increases, the value of the positive solution inside the domain tends to become a constant which is not greater than $\pi$ (the boundary layers appear).

\begin{table}[!ht]
\centering\footnotesize
\caption{Coefficients in the initial guesses of multiple solutions for different $\kappa$ and labels of the corresponding numerical solutions shown in Figs.~\ref{fig:sinG_k6_5sols} and \ref{fig:sinG_k11_10sols}. $n_{\mathrm{sol}}$ denotes the number of solutions to \eqref{eq:sinG-aeq}.}
\label{tab:sinG_a}
\begin{tabular}{cccccccccccc}
\hline
$\kappa$($n_{\mathrm{sol}}$) & $a_1$ & $a_2$ & $a_3$ & $a_4$ & $a_5$ & $a_6$ & $a_7$ & $a_8$ & $a_9$ & $a_{10}$ & Figure \\
\hline
1(0)
& - & - & - & - & - & - & - & - & - & - & \\
\hline
4(2)
& 4.2725 & 0 & 0 & 0 & 0.2644 & 0.2644 & 0 & 0 & 0 & 0 & \\
\hline
\multirow{5}{*}{6(10)}
& 5.2981 & 0 & 0 & 0 & 0.5264 & 0.5264 & 0 & 0 & 0 & 0 & Fig.~\ref{fig:sinG_k6_5sols}(a) \\
& 0 & 2.201 & 0 & 0 & 0 & 0 & 0 & 0.0812 & 0 & 0 & Fig.~\ref{fig:sinG_k6_5sols}(b) \\
& 0 & 0 & 2.201 & 0 & 0 & 0 & 0.0812 & 0 & 0 & 0 & Fig.~\ref{fig:sinG_k6_5sols}(c) \\
& 0 & 1.4358 & 1.4358 & 0 & 0 & 0 & -0.0249 & -0.0249 & 0.033 & 0.033 & Fig.~\ref{fig:sinG_k6_5sols}(d) \\
& 0 & -1.4358 & 1.4358 & 0 & 0 & 0 & -0.0249 & 0.0249 & -0.033 & 0.033 & Fig.~\ref{fig:sinG_k6_5sols}(e) \\
\hline
\multirow{6}{*}{9(12)}
& 6.0381 & 0 & 0 & 0 & 0.819 & 0.819 & 0 & 0 & 0 & 0 & \\
& 0 & 3.9257 & 0 & 0 & 0 & 0 & 0 & 0.4159 & 0 & 0 & \\
& 0 & 0 & 3.9257 & 0 & 0 & 0 & 0.4159 & 0 & 0 & 0 & \\
& 0 & 0 & 0 & 1.7472 & 0 & 0 & 0 & 0 & 0 & 0 & \\
& 0 & 2.574 & 2.574 & 0 & 0 & 0 & -0.1543 & -0.1543 & 0.2076 & 0.2076 & \\
& 0 & -2.574 & 2.574 & 0 & 0 & 0 & -0.1543 & 0.1543 & -0.2076 & 0.2076 & \\
\hline
\multirow{10}{*}{11(20)}
& 6.3239 & 0 & 0 & 0 & 0.9659 & 0.9659 & 0 & 0 & 0 & 0 & Fig.~\ref{fig:sinG_k11_10sols}(a) \\
& 0 & 4.5006 & 0 & 0 & 0 & 0 & 0 & 0.5985 & 0 & 0 & Fig.~\ref{fig:sinG_k11_10sols}(b) \\
& 0 & 0 & 4.5006 & 0 & 0 & 0 & 0.5985 & 0 & 0 & 0 & Fig.~\ref{fig:sinG_k11_10sols}(c) \\
& 0 & 0 & 0 & 2.8345 & 0 & 0 & 0 & 0 & 0 & 0 & Fig.~\ref{fig:sinG_k11_10sols}(d) \\
& 0 & 0 & 0 & 0 & 1.574 & 0 & 0 & 0 & 0 & 0 & Fig.~\ref{fig:sinG_k11_10sols}(e) \\
& 0 & 0 & 0 & 0 & 0 & 1.574 & 0 & 0 & 0 & 0 & Fig.~\ref{fig:sinG_k11_10sols}(f) \\
& 0 & 2.9655 & 2.9655 & 0 & 0 & 0 & -0.2369 & -0.2369 & 0.3262 & 0.3262 & Fig.~\ref{fig:sinG_k11_10sols}(g) \\
& 0 & -2.9655 & 2.9655 & 0 & 0 & 0 & -0.2369 & 0.2369 & -0.3262 & 0.3262 & Fig.~\ref{fig:sinG_k11_10sols}(h) \\
& 0 & 0 & 0 & 0 & 1.0306 & -1.0306 & 0 & 0 & 0 & 0 & Fig.~\ref{fig:sinG_k11_10sols}(i) \\
& -0.0739 & 0 & 0 & 0 & 1.0039 & 1.0039 & 0 & 0 & 0 & 0 & Fig.~\ref{fig:sinG_k11_10sols}(j) \\
\hline
\end{tabular}
\vspace*{1ex}
\end{table}

\begin{figure}[!ht]
  \centering
  \includegraphics[width=.19\textwidth]{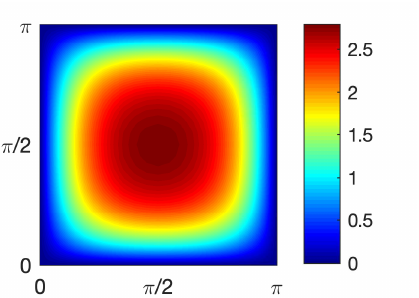}
  \includegraphics[width=.19\textwidth]{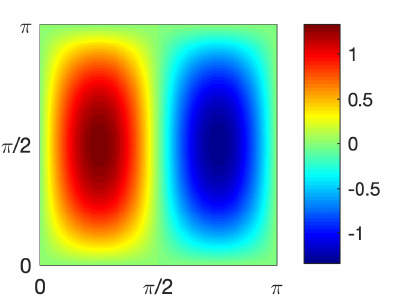}
  \includegraphics[width=.19\textwidth]{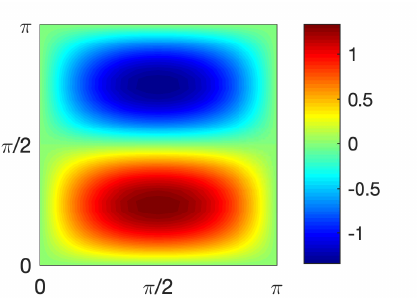}
  \includegraphics[width=.19\textwidth]{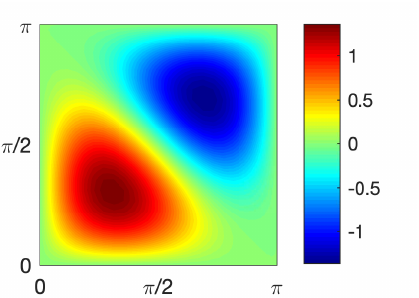}
  \includegraphics[width=.19\textwidth]{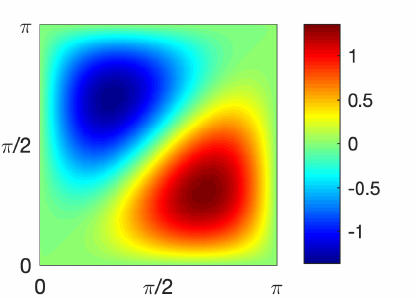} \\[-5pt]
  \hspace*{.06\textwidth}
  (a) \hfill
  (b) \hfill
  (c) \hfill
  (d) \hfill
  (e) \hspace*{.1\textwidth}
  \caption{Five solutions of the problem \eqref{eq:sinG} with $\kappa=6$.}
  \label{fig:sinG_k6_5sols}
\end{figure}

\begin{figure}[!t]
  \centering
  \includegraphics[width=.19\textwidth]{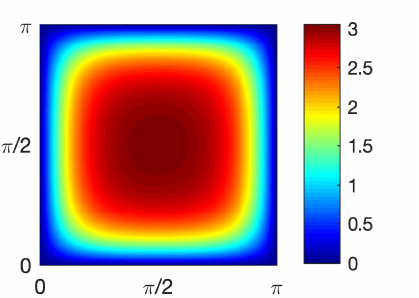}
  \includegraphics[width=.19\textwidth]{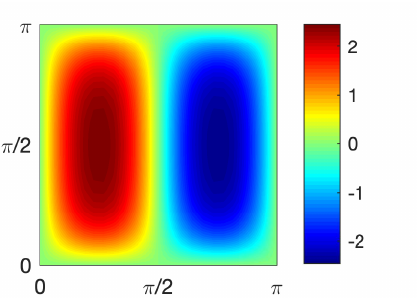}
  \includegraphics[width=.19\textwidth]{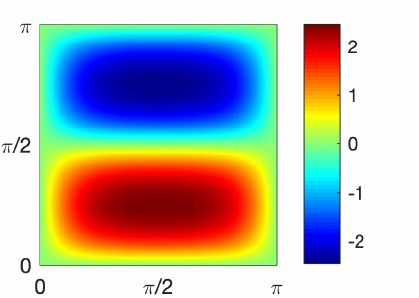}
  \includegraphics[width=.19\textwidth]{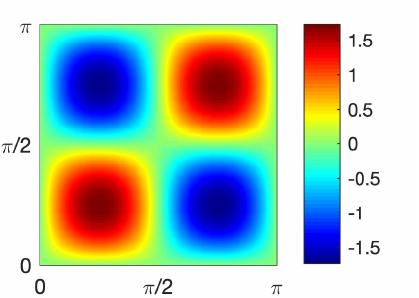}
  \includegraphics[width=.19\textwidth]{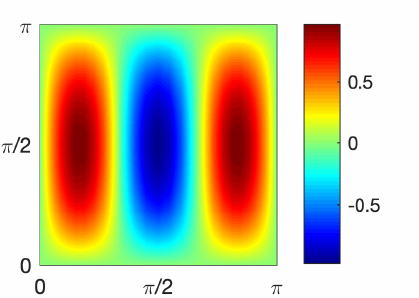} \\[-5pt]
  \hspace*{.06\textwidth}
  (a) \hfill
  (b) \hfill
  (c) \hfill
  (d) \hfill
  (e) \hspace*{.1\textwidth} \\
  \includegraphics[width=.19\textwidth]{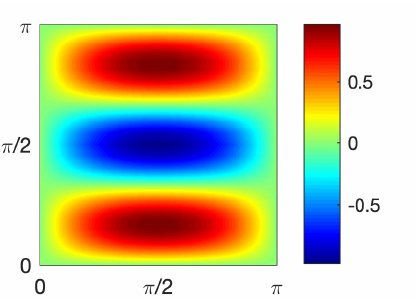}
  \includegraphics[width=.19\textwidth]{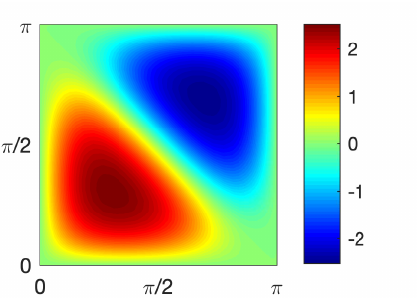}
  \includegraphics[width=.19\textwidth]{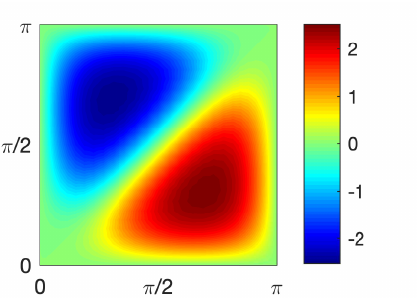}
  \includegraphics[width=.19\textwidth]{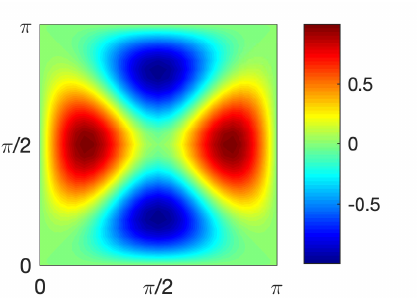}
  \includegraphics[width=.19\textwidth]{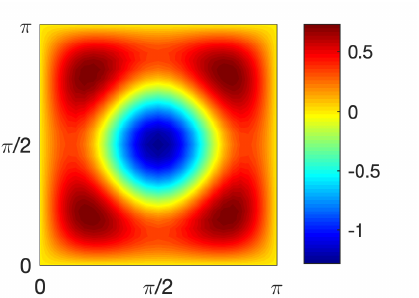} \\[-5pt]
  \hspace*{.06\textwidth}
  (f) \hfill
  (g) \hfill
  (h) \hfill
  (i) \hfill
  (j) \hspace*{.1\textwidth}
  \caption{Ten solutions of the problem \eqref{eq:sinG} with $\kappa=11$.}
  \label{fig:sinG_k11_10sols}
\end{figure}

\section{Conclusion}
\label{sec:5}
Inspired by the traditional SEM, we developed an efficient SGSEM with spectral accuracy for finding multiple solutions to a class of semilinear elliptic Dirichlet boundary value problems. The SGSEM combining with the interpolated coefficient method can reduce the degree of freedom in the numerical computation and make the resulting nonlinear algebraic system very simple and easy to solve, so it is more efficient than the traditional SEM based on the finite element method. More importantly, by employing the Schauder's fixed point theorem and other technical strategies, the existence, uniqueness and spectral convergence for the numerical solutions of the SGSEM corresponding to each specified true solution were rigorously proved. As a result, in any sufficiently small neighborhood of a specified true solution, the SGSEM can obtain a unique numerical solution with the spectral convergence accuracy. It is worthwhile to point out that, our analysis approach can be extended to establish the existence and uniqueness for the numerical solutions of the SEM using ICFEM discretization proposed in \cite{XC} near a specified true solution, which are missed in \cite{XC}. Finally, extensive numerical results for the model problems with cubic and triangular nonlinearities were reported to present multiple solutions with their graphics for visualization and verify the feasibility of our algorithm. Numerical tests on the efficiency and spectral convergence rate of the SGSEM and the numerical comparison of the efficiency of the SEM based on the ICLG and ICFEM discretizations were also provided.

\Acknowledgements{The authors would like to thank the reviewers for their careful review and constructive and valuable comments.}

\end{document}